\documentclass[preprint,12pt,3p]{elsarticle}
\usepackage{color}
\usepackage[colorlinks=true]{hyperref}
\usepackage{amsmath,amssymb,amsthm,amscd}
\usepackage{xcolor}
\numberwithin{equation}{section}

\newtheorem{theorem}{Theorem}[section]

\newtheorem{proposition}{Proposition}[section]
\newtheorem{lemma}{Lemma}[section]
\newtheorem{definition}{Definition}[section]

\newtheorem{remark}{Remark}[section]
\newtheorem{ex}{Example}[section]
\biboptions{sort&compress}
\journal{Elsevier}
\usepackage{times}

\begin{document}

\begin{frontmatter}
\title{Sharp embedding results and geometric inequalities for H\"{o}rmander vector fields}

\author[label1]{Hua Chen\corref{cor1}}
\ead{chenhua@whu.edu.cn}
\author[label2]{Hong-Ge Chen}
\ead{hongge\_chen@whu.edu.cn}
\author[label1]{Jin-Ning Li}
\ead{lijinning@whu.edu.cn}

\address[label1]{School of Mathematics and Statistics, Wuhan University, Wuhan 430072, China}
\address[label2]{Wuhan Institute of Physics and Mathematics, Innovation Academy for Precision Measurement Science and Technology, Chinese Academy of Sciences, Wuhan 430071, China}

\cortext[cor1]{corresponding author}

\begin{abstract}
Let $U$ be a connected open subset of $\mathbb{R}^n$, and let $X=(X_1,X_{2},\ldots,X_m)$ be a system of H\"{o}rmander vector fields defined on $U$. This paper addresses sharp 
 embedding results and geometric inequalities in the generalized Sobolev space $\mathcal{W}_{X,0}^{k,p}(\Omega)$, where $\Omega\subset\subset U$ is a general open bounded subset of $U$. By employing  Rothschild-Stein's lifting technique and saturation method, we prove the representation formula for smooth functions with compact support in $\Omega$. Combining this representation formula with weighted weak-$L^p$ estimates, we derive sharp Sobolev inequalities on $\mathcal{W}_{X,0}^{k,p}(\Omega)$, where the critical Sobolev exponent depends on the generalized M\'{e}tivier index. As applications of these sharp Sobolev inequalities, we establish the isoperimetric inequality, logarithmic Sobolev inequalities, Rellich-Kondrachov compact embedding theorem, Gagliardo-Nirenberg inequality, Nash inequality, and Moser-Trudinger inequality in the context of general H\"{o}rmander vector fields.
\end{abstract}
\end{frontmatter}

\section{Introduction and main results}

For $n\geq 2$, let $U$ be a connected open subset of $\mathbb{R}^n$. Consider a system of real smooth vector fields
$X=(X_1,X_{2},\ldots,X_m)$ defined on $U$, satisfying the following assumption:
\begin{enumerate}
\item [(H)]
 There exists a smallest positive integer $s_{0}$ such that the vector fields $X_1,X_{2},\ldots,X_m$ together with their commutators of length at most $s_{0}$ span $\mathbb{R}^n$ at each point in $U$.
\end{enumerate}
Assumption (H) is known as the H\"{o}rmander's condition, and the positive integer $s_{0}$ is called the H\"{o}rmander index of $X$ related to $U$. The smooth vector fields $X$ under H\"{o}rmander's condition (H) usually refer to H\"{o}rmander vector fields.

Then, we recall the generalized Sobolev spaces associated with $X$. Let
$1\leq j_{i}\leq m$ and $J=(j_{1},\ldots,j_{l})$ denotes a multi-index with length $|J|=l$. We adopt the notation $X^{J}=X_{j_{1}}X_{j_{2}}\cdots X_{j_{l-1}}X_{j_{l}}$ for $|J|=l$, and $X^{J}=\mbox{id}$ for $|J|=0$. For any $k\in \mathbb{N}^{+}$ and $p\geq 1$, we define the function space
\[ \mathcal{W}_{X}^{k,p}(U)=\big\{u\in L^{p}(U)\big|X^{J}u\in L^{p}(U),~~ \forall J=(j_{1},\ldots,j_{s}) ~\mbox{with}~|J|\leq k\big\}, \]
 and set the norm
\[ \|u\|_{\mathcal{W}_{X}^{k,p}(U)}^{p}=\sum_{|J|\leq k}\|X^{J}u\|_{L^p(U)}^{p}.\]
It is well-known (e.g., see \cite[Theorem 1]{Xu90}) that  $\mathcal{W}_{X}^{k,p}(U)$ forms a separable Banach space for $1\leq p<+\infty$, and it is reflexive for $1<p<+\infty$. In particular, we set $\mathcal{H}_{X}^k(U)=\mathcal{W}_{X}^{k,2}(U)$, and it follows that $\mathcal{H}_{X}^k(U)$ is a Hilbert space endowed with the inner product
\begin{equation*}\label{innerproduct}
  (u,v)_{\mathcal{H}_{X}^k(U)}=\sum_{|J|=0}^k\int_{U}X^J u\cdot X^J vdx.
\end{equation*}
Let $\Omega\subset\subset U$ be a bounded open subset. We define  $\mathcal{W}_{X,0}^{k,p}(\Omega)$ as the closure of $C_{0}^{\infty}(\Omega)$ in $\mathcal{W}_{X}^{k,p}(U)$, which gives a generalized Sobolev space of functions vanishing at the boundary of $\Omega$. Similarly, we denote $\mathcal{H}_{X,0}^k(\Omega)$ as $\mathcal{W}_{X,0}^{k,2}(\Omega)$.

Since H\"{o}rmander's celebrated work \cite{hormander1967} on hypoellipticity, the study of nonlinear degenerate elliptic equations arising from H\"{o}rmander vector fields has developed significantly. These nonlinear degenerate equations, taking the form
\begin{equation}\label{1-1}
  \sum_{j=1}^{m}X_{j}^{*}(|Xu|^{p-2}X_{j}u)=f(x,u,Xu)
\end{equation}
with $Xu=(X_1u,X_{2}u,\ldots,X_mu)$ and $X_{j}^{*}=-X_{j}-{\rm div}X_{j}$, naturally arise in the study of the geometry of CR manifolds and in the theory of quasi-conformal mappings on stratified, nilpotent Lie groups, as well as in subelliptic variational problems (see \cite{Jerison-Lee-1988,Jerison-Lee-1989,jost1998,Xu90}). Clearly,  the function space $\mathcal{W}_{X,0}^{k,p}(\Omega)$ is a fundamental space in study the Dirichlet boundary problem of equation \eqref{1-1}. In this context, the sharp embedding results and geometric inequalities on the generalized Sobolev space $\mathcal{W}_{X,0}^{k,p}(\Omega)$ play a pervasive and essential role, similar to their classical counterparts.

 When $X=(\partial_{x_{1}},\partial_{x_{2}},\ldots,\partial_{x_n})$, $\mathcal{W}_{X,0}^{k,p}(\Omega)$ reduces to the classical Sobolev space $W_{0}^{k,p}(\Omega)$. For $k=1$, restricting the Gagliardo-Nirenberg-Sobolev inequality
\begin{equation}\label{1-2}
  \left(\int_{\mathbb{R}^n}|u|^{p^{*}}dx\right)^{\frac{1}{p^{*}}}\leq C\left(\int_{\mathbb{R}^n}|\nabla u|^{p}dx\right)^{\frac{1}{p}}\qquad \forall u\in C_{0}^{\infty}(\mathbb{R}^n)
\end{equation}
on $C_{0}^{\infty}(\Omega)$  naturally yields  the Sobolev embedding
\begin{equation}\label{1-2-a}
W_{0}^{1,p}(\Omega)\hookrightarrow L^{p^{*}}(\Omega),
\end{equation}
where $p^{*}=\frac{np}{n-p}$ is the critical Sobolev exponent depending only on the dimension $n$ and $p$. However, the sharp Sobolev embedding, particularly the critical Sobolev exponent, for the space $\mathcal{W}^{k,p}_{X,0}(\Omega)$ within the framework of H\"{o}rmander vector fields have not been fully understood.

Previous investigations by Capogna-Danielli-Garofalo \cite{Capogna1993,Capogna1994} (also see \cite[Theorem 3.1]{Capogna1996}) claimed the following local Sobolev inequality:
\begin{equation}\label{1-3}
  \left(\frac{1}{|B(x_0,r)|}\int_{B(x_0,r)}|u|^{q}dx\right)^{\frac{1}{q}}\leq Cr\left(\frac{1}{|B(x_0,r)|}\int_{B(x_0,r)}|Xu|^pdx\right)^{\frac{1}{p}}
\end{equation}
for $1\leq p<Q$, $1\leq q\leq \frac{Qp}{Q-p}$, and all $u\in  \mathcal{W}_{X,0}^{1,p}(B(x_0,r))$. Here, $Q$ is the local homogeneous dimension relative to $\overline{\Omega}$; $B(x_0,r)$ denotes the subunit ball induced by the subunit metric $d$, centered at $x_{0}\in \overline{\Omega}$, with a small radius $0<r\leq R_0$, where $R_0$ is a positive constant. The precise definitions of $Q$, $d$ and $B(x,r)$ will be postponed in Section \ref{section2} blow. Based on
\eqref{1-3} and the standard partition of unity arguments, Capogna-Danielli-Garofalo derived the following Sobolev embedding:
\begin{equation}\label{1-4}
  \mathcal{W}_{X,0}^{1,p}(\Omega)\hookrightarrow L^{\frac{Qp}{Q-p}}(\Omega)\qquad \mbox{for}\quad 1\leq p<Q.
\end{equation}
Comparing with the case of left-invariant vector fields on a nilpotent homogeneous group, they announced (see \cite[p. 205]{Capogna1994}) that the Sobolev exponent $\frac{Qp}{Q-p}$ in \eqref{1-3} and \eqref{1-4}, determined by the local homogeneous dimension $Q$, is best possible.

Nevertheless, it is important to note that the vector fields satisfying H\"{o}rmander's condition need not necessarily be the left-invariant vector fields on a nilpotent homogeneous group. Therefore, the Sobolev exponent $\frac{Qp}{Q-p}$ in \eqref{1-3} and \eqref{1-4} may \emph{NOT} be optimal for the space $\mathcal{W}^{1,p}_{X,0}(\Omega)$ associated with general H\"{o}rmander vector fields. Further evidence presented in \cite[Corollary 1]{Yung2015} indicates that if $\Omega\subset\subset U$ is a bounded open subset with smooth boundary $\partial\Omega$, then
\begin{equation}\label{1-5}
  \|f\|_{L^{\frac{p\tilde{\nu}}{\tilde{\nu}-p}}(\Omega)}\leq C\left(\|Xf\|_{L^p(\Omega)}+\|f\|_{L^p(\Omega)}\right)\qquad \forall f\in C^{\infty}(\overline{\Omega}),
\end{equation}
where $1\leq p<\tilde{\nu}$, and $\tilde{\nu}$ denotes the non-isotropic dimension of $\overline{\Omega}$ related to $X$ (see \eqref{2-2} blow for precise definition). The non-isotropic dimension $\tilde{\nu}$ is also known as the generalized M\'{e}tivier index, satisfying $\tilde{\nu} \leq Q$ as shown in \eqref{2-6} below. This implies that for any bounded open subset $\Omega\subset\subset U$ with smooth boundary $\partial\Omega$, \eqref{1-4} can be improved by the following sharper Sobolev embedding:
\begin{equation}\label{1-6}
  \mathcal{W}_{X,0}^{1,p}(\Omega)\hookrightarrow L^{\frac{p\tilde{\nu}}{\tilde{\nu}-p}}(\Omega)\subset L^{\frac{Qp}{Q-p}}(\Omega) \qquad \mbox{for}\quad 1\leq p<\tilde{\nu}.
\end{equation}

The aim of this paper is to explore the sharp embedding results and fundamental geometric inequalities on $\mathcal{W}_{X,0}^{k,p}(\Omega)$ for arbitrary bounded open subsets $\Omega\subset\subset U$, without requiring smoothness on $\partial\Omega$. Our interest in sharp embedding results of generalized Sobolev space is motivated by our recent studies \cite{chen-chen-li-liao2023,chen-chen-yuan2022} on semilinear subelliptic Dirichlet problems, as well as a natural philosophy arising from the classical Sobolev embedding \eqref{1-2-a}. In the classical elliptic case, since the space $W_{0}^{k,p}(\Omega)$ consists of functions that ``\emph{vanishing at the boundary}",  the classical Sobolev embedding \eqref{1-2-a} does not necessitate any smoothness on the boundary $\partial\Omega$. Following this reasoning, it is extremely interesting and important to understand whether \eqref{1-6} can be extended to arbitrary bounded open sets. Furthermore, similar to the classic elliptic case, the sharp embedding in the case where $p>\tilde{\nu}$ certainly plays an essential role in the study of degenerate elliptic equations.

For general  H\"{o}rmander vector fields, one cannot expect the analogous of Gagliardo-Nirenberg-Sobolev inequality \eqref{1-2} to hold in the entire space $\mathbb{R}^n$ without further assumptions. Therefore, the representation formula for functions  $f\in C_{0}^{\infty}(\Omega)$ associated with H\"{o}rmander vector fields serves as a crucial component in establishing the sharp Sobolev embedding in $\mathcal{W}_{X,0}^{k,p}(\Omega)$.  In earlier studies,  such as those cited in \cite[Proposition 2.4]{Capogna1993}, \cite[p. 210, Remark]{Capogna1994} and \cite[Lemma 5.1]{Lu1992}, the construction of this representation formula relies entirely on an ambiguous statement in \cite[p. 114, Corollary]{NSW1985} regarding the global fundamental solution $\Gamma(x,y)$ for the subelliptic operator $-\triangle_{X}=\sum_{j=1}^{m}X_{j}^{*}X_{j}$. Specifically, \cite[p. 114, Corollary]{NSW1985} states that such a fundamental solution $\Gamma(x,y)$ has been constructed in \cite{stein1976} and possesses the local size estimate  in the form of
\begin{equation}\label{1-7}
\Gamma(x,y)\leq C\frac{d(x,y)^2}{|B(x,d(x,y)|}.
\end{equation}
However, as later pointed out by Nagel \cite[Theorem 11]{Nagel1986} and recently emphasized by  Biagi-Bonfiglioli-Bramanti \cite[p. 1882]{Biagi2022}, no fundamental solution is exhibited in \cite{stein1976}; only the parametrix is provided. In fact, the kernel $\Gamma$  obtained by locally saturating the lifted variables for the parametrix $\widetilde{\Gamma}$ associated with the lifting operators  $-\widetilde{\triangle_{X}}=\sum_{j=1}^{m}\widetilde{X_{j}^{*}}\widetilde{X_{j}}$ is only potentially a local parametrix for $-\triangle_{X}$, but \emph{NOT} necessarily a fundamental solution. This means that \eqref{1-7} only provides an estimate for the parametrix $\Gamma$,  \emph{NOT} for the fundamental solution. Additionally, the parametrix $\Gamma$ is defined only locally and in a non-unique way.  Achieving a genuine local fundamental solution by saturating a parametrix in the lifted space requires considerable additional effort.

 To address these technical issues and to ensure the exposition is reasonably self-contained, we provide a rigorous proof of the representation formula for functions $f\in C_{0}^{\infty}(\Omega)$ in Proposition \ref{prop3-2} blow. Drawing inspiration from \cite{stein1976,NSW1985}, our proof employs Rothschild-Stein's lifting technique and saturation method, along with the estimation of type $\lambda$ operators in the high-dimensional lifting space, thus eliminating the dependence on the fundamental solution of $-\triangle_X$. Refining the estimates in \cite{Yung2015}, we then obtain weighted weak-$L^p$ estimates for the $T$ operators induced by the representation formula. By combining these results with carefully estimated volumes of subunit balls and the degenerate Friedrichs-Poincar\'{e} inequality, we extend \eqref{1-6} to arbitrary bounded open subsets $\Omega\subset\subset U$, thereby improving upon the previous results \eqref{1-4}  by  Capogna-Danielli-Garofalo \cite{Capogna1993,Capogna1994}. Specifically, we have

\begin{theorem}[Sobolev inequalites for $kp\leq \tilde{\nu}$]
\label{thm1}
Let $X=(X_1,X_{2},\ldots,X_m)$ satisfy condition (H). Then, for any bounded open subset $\Omega\subset\subset U$ and any positive number $p\geq 1$, there exists a positive constant $C>0$ such that
\begin{enumerate}[(1)]
  \item If $kp< \tilde{\nu}$ and  $\frac{1}{q}=\frac{1}{p}-\frac{k}{\tilde{\nu}}$, we have
  \begin{equation}\label{1-8}
\|u\|_{L^{q}(\Omega)}\leq C\sum_{|J|=k}\|X^{J}u\|_{L^{p}(\Omega)}\qquad \forall u\in \mathcal{W}^{k,p}_{X,0}(\Omega);
      \end{equation}
  \item  If $kp=\tilde{\nu}$ and  $1\leq q<\infty$, we also have
  \begin{equation}\label{1-9}
      \|u\|_{L^{q}(\Omega)}\leq C\sum_{|J|=k}\|X^{J}u\|_{L^{p}(\Omega)}\qquad \forall u\in \mathcal{W}^{k,p}_{X,0}(\Omega).
      \end{equation}
\end{enumerate}
Here, $\tilde{\nu}$ is the generalized M\'{e}tivier index defined in \eqref{2-2} below.
\end{theorem}

Theorem \ref{thm1} yields the following fundamental inequalities:
\begin{theorem}
\label{thm2}
Suppose $X=(X_1,X_{2},\ldots,X_m)$ satisfy condition (H), and $\Omega\subset\subset U$ is a bounded open subset of $U$. Let $p\geq 1$ be a positive real numbers, $k\in\mathbb{N}^{+}$ with $kp<\tilde{\nu}$.  Assume further that the positive real numbers $s_{1},s_{2},a,b$ satisfy the following conditions:
\begin{equation}\label{conditions}
    1\leq s_{1}<s_{2}<\frac{p\tilde{\nu}}{\tilde{\nu}-kp}~~~~\text{and}~~~~\left\{
         \begin{array}{ll}
           b-a=\frac{\tilde{\nu}-kp}{\tilde{\nu}},\\[2mm]
          bs_{2}-as_{1}=p,
         \end{array}
       \right.
  \end{equation}
  Then, there exists a positive constant $C>0$ such that
    \begin{equation}\label{1-11}
\|u\|_{L^{s_{2}}(\Omega)}^{bs_{2}}\leq C\sum_{|J|= k}\|X^J u\|_{L^p(\Omega)}^{p}\|u\|_{L^{s_{1}}(\Omega)}^{as_{1}},~~~~\forall u\in \mathcal{W}_{X,0}^{k,p}(\Omega)\cap L^{s_{1}}(\Omega).
\end{equation}
As a result of \eqref{1-11}, we have
\begin{enumerate}[(1)]
\item (Gagliardo-Nirenberg inequality) If $s_{1}\geq 1$, $b=\frac{p}{\theta s_{2}}$ and $a=\frac{p(1-\theta)}{s_{1}\theta}$ with $\theta\in(0,1]$, then $\frac{1}{s_{2}}=\frac{1-\theta}{s_{1}}+\theta\left(\frac{1}{p}-\frac{k}{\tilde{\nu}}\right)$ and
    \[\|u\|_{L^{s_{2}}(\Omega)}\leq C\|u\|_{L^{s_{1}}(\Omega)}^{1-\theta}\sum_{|J|= k}\|X^J u\|_{L^p(\Omega)}^{\theta}\qquad \forall u\in \mathcal{W}_{X,0}^{k,p}(\Omega)\cap L^{s_{1}}(\Omega).\]

      \item (Nash inequality) If $\tilde{\nu}>2$, $p=2$, $k=1$, $s_{1}=1$, $s_{2}=2$, $a=\frac{4}{\tilde{\nu}}$ and $b=1+\frac{2}{\tilde{\nu}}$, then
      \[\left(\int_{\Omega}|u|^2 dx\right)^{1+\frac{2}{\tilde{\nu}}}\leq  C\left(\int_{\Omega}|Xu|^2dx\right)\left(\int_{\Omega}|u| dx\right)^{\frac{4}{\tilde{\nu}}}\quad\forall u\in \mathcal{H}_{X,0}^{1}(\Omega).\]
      \item (Moser inequality) If $\tilde{\nu}>2$, $p=2$, $k=1$, $s_{1}=2$, $s_{2}=2+\frac{4}{\tilde{\nu}}$, $a=\frac{2}{\tilde{\nu}}$ and $b=1$, then
\[\left(\int_{\Omega}|u|^{2+\frac{4}{\tilde{\nu}}} dx\right)\leq  C\left(\int_{\Omega}|Xu|^2dx\right)\left(\int_{\Omega}|u|^2 dx\right)^{\frac{2}{\tilde{\nu}}}\quad\forall u\in \mathcal{H}_{X,0}^{1}(\Omega).\]
\end{enumerate}
\end{theorem}

We next discuss the isoperimetric inequality for Carnot-Carath\'{e}odory spaces. For this purpose, we introduce some notations in \cite{Monti2001,Garofalo1996}. For $u\in L^{1}(\Omega)$, we define the $X$-variation of $u$ in $\Omega$ as
\begin{equation}\label{def-var}
{\rm Var}_{X}(u;\Omega)=\sup_{\varphi\in \mathcal{F}(\Omega;\mathbb{R}^m)}\int_{\Omega}u\sum_{j=1}^{m}X_{j}^{*}\varphi_{j}dx,
\end{equation}
where $$\mathcal{F}(\Omega;\mathbb{R}^m):=\left\{\varphi=(\varphi_{1},\ldots,\varphi_{m})\in (C_{0}^{1}(\Omega))^{m}\bigg|\|\varphi\|_{L^{\infty}(\Omega)}=\sup_{x\in \Omega}\left(\sum_{j=1}^{m}|\varphi_{j}(x)|^{2}\right)^{\frac{1}{2}}\leq 1\right\},$$ and $X_{j}^{*}$ denotes the formal adjoint of $X_{j}$. The divergence theorem easily gives that ${\rm Var}_{X}(u;\Omega)=\|Xu\|_{L^{1}(\Omega)}$ for all $u\in C^{1}(\Omega)\cap \mathcal{W}_{X}^{1,1}(\Omega)$. Given a measurable set $E\subset \mathbb{R}^n$,  we denote by
\begin{equation}\label{def-px}
 P_{X}(E;\Omega):={\rm Var}_{X}(\chi_{E};\Omega)=\sup_{\varphi\in \mathcal{F}(\Omega;\mathbb{R}^m)}\int_{E}\sum_{j=1}^{m}X_{j}^{*}\varphi_{j}dx
\end{equation}
 the $X$-perimeter of $E$ in $\Omega$,  where  $\chi_{E}$ is the indicator function of $E$. The set $E$ is of finite $X$-perimeter (or a $X$-Caccioppoli set) in $\Omega$ if $P_{X}(E;\Omega)<+\infty$.

 By means of the Sobolev inequality \eqref{1-3}, Capogna-Danielli-Garofalo \cite{Capogna1994,Capogna1994-iso} initially constructed the isoperimetric inequality
 in subunit balls. Specifically, they  proved that for any subunit ball $B(x_0,r)$  centered at $x_{0}\in \overline{\Omega}$, with a small radius $0<r\leq R_0$, there exists a positive constant $C>0$ such that
 \begin{equation}\label{iso-GDG}
   |E|^{\frac{Q-1}{Q}}\leq Cr|B(x_0,r)|^{-\frac{1}{Q}}P_{X}(E;B(x_0,r))
 \end{equation}
holds for every $C^{1}$ open set $E\subset \overline{E}\subset B(x_0,r)$, where  $R_0$ is a positive constant as mentioned above. Later, Garofalo-Nhieu \cite[Theorem 1.18]{Garofalo1996} extended
\eqref{iso-GDG} to general X-PS domains, which shows that for any X-PS domain $\Omega\subset \mathbb{R}^n$,
\begin{equation}\label{iso-G-D}
  \min(|E\cap \Omega|, |E^{c}\cap\Omega|)^{\frac{Q-1}{Q}}\leq C\text{diam}(\Omega)|\Omega|^{-\frac{1}{Q}}P_{X}(E;\Omega)
\end{equation}
holds for any $X$-Caccioppoli set $E\subset \mathbb{R}^n$.  It is worth  mentioning that the isoperimetric exponent $\frac{Q-1}{Q}$
in both \eqref{iso-GDG} and \eqref{iso-G-D} depend on the local homogeneous dimension $Q$. However,  according to Theorem \ref{thm1}, we can derive a new type of isoperimetric inequality equipped with isoperimetric exponent $\frac{\tilde{\nu}-1}{\tilde{\nu}}$, depending only on the generalized M\'{e}tivier index $\tilde{\nu}$, rather than being dependent on the  local homogeneous dimension $Q$. For more results related to the isoperimetric inequality, one can refer to \cite{Capogna-Danielli-Pauls2007,Chavel2001,Croke1980,Capogna1994-iso,Biroli1995,Monti2004-iso} and the references therein.

\begin{theorem}[Isoperimetric inequality]
\label{thm3}
Assume $X=(X_1,X_{2},\ldots,X_m)$ satisfy condition (H), and $\Omega\subset\subset U$ is a bounded open subset of $U$. Then, there exists a positive constant $C>0$ such that for any bounded open set $E\subset\subset \Omega$ with $C^{1}$ boundary $\partial E$, we have
\begin{equation}\label{1-14}
 |E|^{\frac{\tilde{\nu}-1}{\tilde{\nu}}}\leq CP_{X}(E;\Omega).
\end{equation}
\end{theorem}

The logarithmic Sobolev inequality, originally introduced by Gross \cite{Gross1975} in Euclidean space with the Gaussian measure, is closely related to many important properties of the corresponding Markov semigroup.
According to Theorem \ref{thm1} and  \cite{Davies1990}, we can construct the following logarithmic Sobolev inequalities in the context of H\"{o}rmander vector fields. More results on logarithmic Sobolev inequalities can be found in \cite{Gordina2022,Ruzhansky2021,Guionnet2003}.

\begin{theorem}[Logarithmic Sobolev inequalities]
\label{thm4}
Suppose that $X=(X_1,X_{2},\ldots,X_m)$ and $\Omega$ satisfy the assumptions of Theorem \ref{thm1}. Then
\begin{enumerate}[(1)]
  \item For any $\varepsilon>0$ and $u\in \mathcal{H}_{X,0}^1(\Omega)\cap L^\infty(\Omega)$, we have
\begin{equation}\label{1-15}
\int_{\Omega}|u|^2\ln |u|dx\leq \varepsilon\|Xu\|_{L^2(\Omega)}^2+M(\varepsilon)\|u\|_{L^2(\Omega)}^2+\|u\|_{L^2(\Omega)}^2\ln\|u\|_{L^2(\Omega)},
\end{equation}
where $M(\varepsilon)=\ln C_0-\frac{\tilde{\nu}}{4}\ln \varepsilon$, and $C_0>0$ is a positive constant defined in \eqref{4-21} blow.
  \item Suppose that $p\geq1$,  $k\in \mathbb{N}^{+}$ with $kp\leq \tilde{\nu}$, then for any $u\in \mathcal{W}^{k,p}_{X,0}(\Omega)$ with $\|u\|_{L^p(\Omega)}=1$, we have
\begin{equation}\label{1-16}
\frac{C}{e}\sum_{|J|=k}\|X^J u\|_{L^p(\Omega)}\geq \ln\left( C\sum_{|J|=k}\|X^J u\|_{L^p(\Omega)}\right)\geq \frac{k}{\tilde{\nu}}\int_{\Omega}|u|^p\ln|u|^pdx,
\end{equation}
  where $C>0$ is the positive constant in Theorem \ref{thm1}.
\end{enumerate}
\end{theorem}

Additionally, we  are concerned with the Sobolev inequality in the case of $p>\frac{\tilde{\nu}}{k}$. To proceed, we introduce the generalized H\"{o}lder spaces associated with the subunit metric $d$. For any $0<\alpha<1$ and $u\in C(\overline{\Omega})$, we
 define
\[[u]_{\alpha}:=\sup_{x\neq y,~x,y\in\Omega}\frac{|u(x)-u(y)|}{d(x,y)^\alpha}.\]
Then, the generalized H\"{o}lder spaces are given by
\[ \mathcal{S}^{0,\alpha}(\overline{\Omega}):=\{u\in C(\overline{\Omega})| [u]_{\alpha}<+\infty\}, \]
and
\[ \mathcal{S}^{k,\alpha}(\overline{\Omega}):=\{u\in C(\overline{\Omega})| X^J u\in \mathcal{S}^{0,\alpha}(\overline{\Omega}) ,~\forall |J|\leq k\}\qquad \mbox{for}~~k\in \mathbb{N}.\]
 Note that $\mathcal{S}^{k,\alpha}(\overline{\Omega})$ is a Banach space equipped with the norm
\[\|u\|_{\mathcal{S}^{k,\alpha}(\overline{\Omega})}:=\sum_{|J|\leq k}\left(\sup_{x\in\Omega}|X^{J} u(x)|+[X^J u]_{\alpha}\right).\]
In particular, if $X=(\partial_{x_{1}},\partial_{x_{2}},\ldots,\partial_{x_{n}})$,  $\mathcal{S}^{k,\alpha}(\overline{\Omega})$ reduces to the classical H\"{o}lder space $C^{k,\alpha}(\overline{\Omega})$. In fact, $\mathcal{S}^{k,\alpha}(\overline{\Omega})$ and $C^{k,\alpha}(\overline{\Omega})$ has the following relationship (see \cite[Theorem 1.53]{Bramanti2022}):
\begin{equation}
  C^{0,\alpha}(\overline{\Omega})\subset \mathcal{S}^{0,\alpha}(\overline{\Omega})\subset C^{0,\frac{\alpha}{s_{0}}}(\overline{\Omega}).
\end{equation}
Besides, for every $k\in \mathbb{N}^{+}$ we have
\begin{equation}
 C^{k,\alpha}(\overline{\Omega})\subset\mathcal{S}^{k,\alpha}(\overline{\Omega}),~~~\mbox{and}~~~ \mathcal{S}^{ks_{0},\alpha}(\overline{\Omega})\subset  C^{k,\frac{\alpha}{s_{0}}}(\overline{\Omega}).
\end{equation}
Here, $s_0$ denotes the H\"{o}rmander index of $U$.

An earlier study by  Garofallo-Nhieu \cite[Theorem 1.10]{Garofallo-Nhieu-1998} yields that, for any $(\varepsilon,\delta)$ domain $\Omega\subset\subset U$,  the Sobolev embedding
\begin{equation}\label{holder-embedding-1998}
  \mathcal{W}_{X}^{1,p}(\Omega)\hookrightarrow \mathcal{S}^{0,1-\frac{Q}{p}}(\overline{\Omega})
\end{equation}
holds for $p>Q$. Our next result indicates that, if we consider the Sobolev embedding for the space $\mathcal{W}_{X,0}^{1,p}(\Omega)$, the assumption of $(\varepsilon,\delta)$ domain can be removed, and the Sobolev exponent can be improved to $1-\frac{\tilde{\nu}}{p}$.

\begin{theorem}[Sobolev inequality for $kp>\tilde{\nu}$]
\label{thm5}
Assume that $X=(X_1,X_{2},\ldots,X_m)$ and $\Omega$ satisfy the assumptions of Theorem \ref{thm1}. Let $k\in \mathbb{N}^{+}$ and $p\geq 1$ be a positive number such that $p>\frac{\tilde{\nu}}{k}$. Then there exists a positive constant $C>0$ such that
\begin{equation}\label{holder-embedding}
  \|u\|_{\mathcal{S}^{k-\left[\frac{\tilde{\nu}}{p}\right]-1,\alpha}(\overline{\Omega})}\leq
C\sum_{|J|=k}\|X^J u\|_{L^p(\Omega)}\qquad\forall u\in \mathcal{W}_{X,0}^{k,p}(\Omega),
\end{equation}
where
\begin{equation*}
\alpha=\left\{
            \begin{array}{ll}
              \left[\frac{\tilde{\nu}}{p}\right]+1-\frac{\tilde{\nu}}{p}, & \hbox{if $\frac{\tilde{\nu}}{p}\notin \mathbb{N}$;} \\[2mm]
              \text{any real number in $(0,1)$},  & \hbox{if $\frac{\tilde{\nu}}{p}\in \mathbb{N}$,}
            \end{array}
          \right.
\end{equation*}
\end{theorem}

\begin{remark}
According to Theorem \ref{thm1} and Theorem \ref{thm6}, we conclude that for $k\geq 1$,
\begin{equation}\label{1-20}
  \mathcal{W}_{X,0}^{k,p}(\Omega)\hookrightarrow \left\{
                                                   \begin{array}{ll}
                                                     L^{q}(\Omega) , & \hbox{if $1\leq p<\frac{\tilde{\nu}}{k}$ and $1\leq q\leq \frac{\tilde{\nu}p}{\tilde{\nu}-pk}$;} \\[2mm]

                                                     L^{q}(\Omega), & \hbox{if $p=\frac{\tilde{\nu}}{k}$ and $1\leq q<\infty$;}\\[2mm] \mathcal{S}^{k-\left[\frac{\tilde{\nu}}{p}\right]-1,\alpha}(\overline{\Omega}), & \hbox{if $p\geq 1$ and $p>\frac{\tilde{\nu}}{k}$.}
                                                   \end{array}
                                                 \right.
\end{equation}
In particular, for $k=1$ we have
\begin{equation}\label{1-21}
  \mathcal{W}_{X,0}^{1,p}(\Omega)\hookrightarrow   \left\{
                                                   \begin{array}{ll}
                                                     L^{q}(\Omega) , & \hbox{if $1\leq p<\tilde{\nu}$ and $1\leq q\leq \frac{\tilde{\nu}p}{\tilde{\nu}-p}$;} \\[2mm]

                                                     L^{q}(\Omega), & \hbox{if $p=\tilde{\nu}$ and $1\leq q<\infty$;}\\[2mm] \mathcal{S}^{0,\alpha}(\overline{\Omega}), & \hbox{if $p\geq 1$ and $p>\tilde{\nu}$.}
                                                   \end{array}
                                                 \right.
\end{equation}

As a consequence of \eqref{1-21}, we see that the Sobolev embedding \eqref{1-6} indeed holds for any bounded open set $\Omega\subset\subset U$. According to \eqref{2-6} and Example \ref{example2-1} below, we find that $\tilde{\nu}\leq Q$, and $\tilde{\nu}<Q$ occurs in some degenerate cases. Therefore, our embedding result \eqref{1-21}  genuinely improves upon \eqref{1-4}, as established by Capogna-Danielli-Garofalo \cite{Capogna1993,Capogna1994}.
\end{remark}

Meanwhile, we can obtain the following degenerate Rellich-Kondrachov compact embedding theorem, which generalizes our previous result \cite[Proposition 2.7]{chen-chen-yuan2022} and  may shed light on the study of degenerate equations and subelliptic variational problems.

\begin{theorem}[Rellich-Kondrachov compact embedding theorem]
\label{thm6}
Let $X=(X_1,X_{2},\ldots,X_m)$ satisfy condition (H). Suppose that $\Omega\subset\subset U$ is a bounded open subset of $U$. Then, for $p\geq 1$, $k\in \mathbb{N}^{+}$ with $kp<\tilde{\nu}$, the embedding
\begin{equation}
  \mathcal{W}_{X,0}^{k,p}(\Omega)\hookrightarrow L^{s}(\Omega)
\end{equation}
is compact for $1\leq s<\frac{\tilde{\nu}p}{\tilde{\nu}-kp}$.
\end{theorem}

Finally, we can also present the Moser-Trudinger inequality for $\mathcal{W}^{1,\tilde{\nu}}_{X,0}(\Omega)$.  It is noteworthy that the corresponding Moser-Trudinger inequality for the Sobolev space $\mathcal{W}^{1,Q}_{X}(\Omega)$ associated with the X-PS domain $\Omega\subset\subset U$ was established in \cite{Garofalo1996}. Additionally, the Moser-Trudinger inequality on Carnot groups has been explored in \cite{Lu-Cohn2001,Lam2014,Balogh2003}, among others.

\begin{theorem}[Moser-Trudinger inequality]
\label{thm7}
Suppose $X=(X_1,X_{2},\ldots,X_m)$ and $\Omega$ satisfy the assumptions of Theorem \ref{thm1}. Moreover, assume that for every $x\in U$ and $r>0$, the subunit ball $B(x,r)$ admits finite volume, i.e.,
\begin{equation}
  |B(x,r)|<+\infty\qquad \forall x\in U,~r>0.  \tag{A}
\end{equation}
 Then for any
\[ 0<\sigma<\frac{\tilde{\nu}-1}{e\tilde{\nu}}C_{3}^{-\frac{\tilde{\nu}+1}{\tilde{\nu}-1}}\left( \frac{C_{1}}{4C_{0}(C_{1}+C_{2})}\right)^{\frac{\tilde{\nu}}{\tilde{\nu}-1}}\cdot\left(\frac{\lambda_{1}(\tilde{\nu})CC_{1}}{\lambda_{1}(\tilde{\nu})+1} \right)^{\frac{1}{\tilde{\nu}-1}}, \]
there exists $\widetilde{C}>0$ such that for any $u\in \mathcal{W}^{1,\tilde{\nu}}_{X,0}(\Omega)$ with $\|Xu\|_{L^{\tilde{\nu}}(\Omega)}\leq 1$, we have
\begin{equation}
 \int_{\Omega}e^{\sigma|u|^{\frac{\tilde{\nu}}{\tilde{\nu}-1}}}dx\leq \widetilde{C}|\Omega|.
\end{equation}
Here, $C_{0}>0$ is a positive constant given by \eqref{4-33}, $C,C_{1},C_{2},C_{3}$ are the positive constants appeared in Proposition \ref{prop2-2}-Proposition \ref{prop2-4}, and
\[ \lambda_{1}(\tilde{\nu}):=\inf_{u\in \mathcal{W}_{X,0}^{1,\tilde{\nu}}(\Omega),~u\neq 0}\frac{\int_{\Omega}|Xu|^{\tilde{\nu}}dx}{\int_{\Omega}|u|^{\tilde{\nu}}dx}>0.\]
\end{theorem}
\begin{remark}
From \cite[Remark 2.5]{Garofallo-Nhieu-1998}, we see that for every compact subset $K\subset U$, there exists $R_0>0$ such that the subunit balls $B(x,r)$, with $x\in K$ and $0<r\leq R_0$, are compact. However, this property does not generally hold for large radius. The assumption (A) serves as a necessary condition for the upper bound volume estimates of subunit balls with large radius (see Proposition \ref{prop2-3} and Proposition \ref{prop2-4} below). We point out
that the class of H\"{o}rmander vector fields under assumption (A) is quite large.

For instance, considering $U=\mathbb{R}^n$ and
 $X=(X_{1},X_{2},\ldots,X_{m})$  satisfying (H) along with the following homogeneity assumption:
\begin{enumerate}
  \item [(H.1)]

   There exists a family of non-isotropic dilations $\{\delta_{t}\}_{t>0}$ of the form
  \[ \delta_{t}:\mathbb{R}^n\to \mathbb{R}^n,\qquad \delta_{t}(x)=(t^{\sigma_{1}}x_{1},t^{\sigma_{2}}x_{2},\ldots,t^{\sigma_{n}}x_{n}), \]
 where $1=\sigma_{1}\leq \sigma_{2}\leq\cdots\leq \sigma_{n}$ are positive integers, such that $X_{1},X_{2},\ldots,X_{m}$ are $\delta_{t}$-homogeneous of degree $1$. That is, for all $t>0$, $f\in  C^{\infty}(\mathbb{R}^n)$, and $j = 1, \ldots, m$,
 \[ X_{j}(f\circ \delta_{t})=t(X_{j}f)\circ \delta_{t}. \]
 \end{enumerate}
In this case, $X$ are the so-called homogeneous H\"{o}rmander vector fields, and the assumption (A) is derived from the global version ball-box theorem in \cite[Theorem B]{Biagi2022}.

Moreover, if every vector field $X_{i}=\sum_{k=1}^{n}b_{ik}(x)\partial_{x_{k}}$ satisfies $b_{ik}\in L^{\infty}(U)$, then according to \cite[Proposition 1.37]{Bramanti2022} we have $d(x,y)\geq C|x-y|$ for all $x,y\in U$, and consequently,
\begin{equation*}
  B(x,r)\subset B_{E}\left(x,\frac{r}{C}\right)\cap U\qquad \mbox{for any}~~x\in U,~r>0,
\end{equation*}
where $B_{E}(x,r)=\{y\in\mathbb{R}^n| |x-y|<r\}$ denotes the classical Euclidean ball in $\mathbb{R}^n$. Hence, assumption (A) holds true in this case as well.

\end{remark}

The rest of the paper is organized as follows.  In Section \ref{section2}, we present some necessary preliminaries,
 including the comparison of local homogeneous dimension and generalized M\'{e}tivier index, the subunit metric and volume estimates of subunit balls, the degenerate Friedrichs-Poincar\'{e} type inequality, and the chain rules in general Sobolev spaces. In Section \ref{section3}, we then provide two different types of the representation formulas and construct the weighted weak-$L^p$ estimates of the $T$ operators induced by the representation formula. Finally, we prove Theorem \ref{thm1}-Theorem \ref{thm7} in Section \ref{section4}.

\emph{\textbf{Notations}}. For the sake of simplicity, different positive constants are usually denoted by $C$ sometimes
without indices.

\section{Preliminaries}
\label{section2}
We begin with the some basic objects and notations in Carnot-Carath\'{e}odory space.
\subsection{Basic objects and notations in Carnot-Carath\'{e}odory space}

Let $\mbox{Lie}(X)$ be the Lie algebra generated by vector fields $X_{1},X_{2},\ldots,X_{m}$ over $\mathbb{R}$.
For $l\in \mathbb{N}^{+}$, we define
\begin{equation*}
  \mbox{Lie}^{l}(X):=\mbox{span}\{[X_{i_{1}},\ldots,[X_{i_{j-1}},X_{i_{j}}]]|1\leq i_{j}\leq m, j\leq l\}.
\end{equation*}
The H\"{o}rmander's condition (H) gives that $\mbox{Lie}(X)(x)=\{Z(x)|Z\in \mbox{Lie}(X)\}=T_{x}(U)$ for all $x\in U$. This means, for each point $x\in U$, there exists a minimal integer $s(x)\leq s_{0}$ such that
\begin{equation*}
  \mbox{Lie}^{s(x)}(X)(x):=\{Z(x)|Z\in \mbox{Lie}^{s(x)}(X)\}=T_{x}(U).
\end{equation*}
The integer $s(x)$ is known as the degree of nonholonomy at $x$.

For $x\in U$ and $1\leq j\leq s(x)$, we set $V_{j}(x):=\mbox{Lie}^{j}(X)(x)$. It follows that
\[ \{0\}=V_{0}(x)\subset V_{1}(x)\subset  \cdots \subset V_{s(x)-1}(x)\subsetneq  V_{s(x)}(x)= T_{x}(U).\] Then, we define
\begin{equation}\label{2-1}
\nu(x):=\sum_{j=1}^{s(x)}j(\nu_{j}(x)-\nu_{j-1}(x))
\end{equation}
as the pointwise homogeneous dimension at $x$ (see \cite{Morbidelli2000}), where $\nu_{j}(x):=\mbox{dim}V_{j}(x)$ with $\nu_{0}(x):=0$.  Note that \eqref{2-1} implies $n\leq n+s(x)-1\leq \nu(x)\leq ns(x)$.

We say a point $x\in U$ is regular if, for every $1\leq j\leq s(x)$, the dimension $\nu_{j}(y)$ is a constant as $y$ varies in an open neighbourhood of $x$. Otherwise, $x$ is said to be singular.
 Moreover, for any subset $\Omega \subset\subset U$, we say $\Omega$ is equiregular if every point of $\Omega$ is regular, while  $\Omega$ is said to be non-equiregular if it contains some singular points. The equiregular assumption is also known as the M\'etivier's condition in PDEs (see \cite{Metivier1976}). For the equiregular connected subset $\Omega$, the pointwise homogeneous dimension $\nu(x)$ is a constant $\nu$ consistent with the Hausdorff dimension of $\Omega$ with respect to $X$, and this constant $\nu$ is also called the M\'etivier index.  Additionally, if the subset $\Omega\subset U$ is non-equiregular, we can introduce the generalized M\'{e}tivier index  by
\begin{equation}\label{2-2}
  \tilde{\nu}:=\max_{x\in \overline{\Omega}}\nu(x).
\end{equation}
The generalized M\'{e}tivier index  is
also called the non-isotropic dimension (see \cite{Yung2015,chen-chen2019,chen-chen2020}), which plays an important role in the geometry and functional settings associated
with vector fields $X$.

Then, we introduce some notations in \cite{NSW1985} to present the precise definition of local homogeneous dimension.
 Let $J=(j_{1},\ldots,j_{k})$ be a multi-index with length $|J|=k$, where $1\leq j_{i}\leq m$ and $1\leq k\leq r$. We assign a commutator $X_{J}$ of length $k$  such that
 \[  X_{J}:=[X_{j_{1}},[X_{j_{2}},\ldots[X_{j_{k-1}},X_{j_{k}}]\ldots]],\]
and set
\[ X^{(k)}:=\{X_{J}|J=(j_{1},\ldots,j_{k}),~1\leq j_{i}\leq m, |J|=k \} \]
 the collection of all commutators of length $k$.  Let $Y_{1},\ldots,Y_{l}$ be an enumeration of the components of $X^{(1)},\ldots,X^{(s_{0})}$. We say $Y_{i}$ has formal degree $d(Y_{i})=k$ if $Y_{i}$ is an element of $X^{(k)}$.
 If $I=(i_{1},i_{2},\ldots,i_{n})$ is a $n$-tuple of integers with $1\leq i_{k}\leq l$, we define \[ d(I):=d(Y_{i_{1}})+d(Y_{i_{2}})+\cdots+d(Y_{i_{n}}), \]
  and the so-called Nagel-Stein-Wainger polynomial
\begin{equation}\label{2-3}
  \Lambda(x,r):=\sum_{I}|\lambda_{I}(x)|r^{d(I)},
\end{equation}
where $\lambda_{I}(x):=\det(Y_{i_{1}},Y_{i_{2}},\ldots,Y_{i_{n}})(x)$, and $I=(i_{1},i_{2},\ldots,i_{n})$ ranges in the set of $n$-tuples satisfying $1\leq i_{k}\leq l$.

We now recall the local homogeneous dimension introduced by Capogna-Danielli-Garofalo in \cite{Capogna1993,Capogna1994,Capogna1996}. Let $\Omega\subset\subset U$ be a bounded open set. According to
 \cite[(3.4), p. 1166]{Capogna1996} and \cite[(3.1), p. 105]{Garofalo2018}, the local homogeneous dimension $Q$ relative to the bounded set $\overline{\Omega}$ is precisely defined as follows:
\begin{equation}\label{2-4}
 Q:=\max\{d(I)|\lambda_{I}(x)\neq 0~\mbox{and}~x\in \overline{\Omega}\}=\sup_{x\in\overline{\Omega}}\left(\lim_{r\to+\infty}\frac{\ln \Lambda(x,r)}{\ln r}\right).
\end{equation}

To compare the local homogeneous dimension and generalized M\'{e}tivier index, we employ the following proposition.
\begin{proposition}[{\cite[Proposition 2.2]{chen-chen2019}}]
\label{prop2-1}
For each $x\in U$, the pointwise homogeneous dimension $\nu(x)$ can be characterized by
\begin{equation}\label{2-5}
  \nu(x)=\sum_{j=1}^{s(x)}j(\nu_{j}(x)-\nu_{j-1}(x))=\lim_{r\to 0^{+}}\frac{\ln\Lambda(x,r)}{\ln r}=\min\{d(I)|\lambda_{I}(x)\neq 0\}.
\end{equation}
\end{proposition}

Obviously, by \eqref{2-2}, \eqref{2-4} and \eqref{2-5} we get
\begin{equation}\label{2-6}
 \tilde{\nu}=\max_{x\in \overline{\Omega}}\nu(x)=\max_{x\in \overline{\Omega}}\left(\min\{d(I)|\lambda_{I}(x)\neq 0\}\right)\leq \max_{x\in \overline{\Omega}}\left(\max\{d(I)|\lambda_{I}(x)\neq 0\}\right)= Q.
\end{equation}
It is worth mentioning that $\tilde{\nu}$ is strictly less than $Q$ in some degenerate cases. For example,
\begin{ex}
\label{example2-1}
Let $X=(X_1,X_2,X_3)=(e^{x_2}\partial_{x_1}, e^{2x_2}\partial_{x_1}, x_1\partial_{x_2})$ be the smooth vector fields defined in $\mathbb{R}^{2}$. Assume that $\Omega=\{x\in \mathbb{R}^2| |x|<1\}$ is the unit ball in $\mathbb{R}^{2}$. A direct calculation yields that
\[ [X_{1},X_{2}]=0,\quad [X_{1},X_{3}]=e^{x_{2}}\partial_{x_{2}}-x_{1}e^{x_{2}}\partial_{x_{1}},\quad \mbox{and}~~~ [X_{2},X_{3}]=e^{2x_{2}}\partial_{x_{2}}-2x_{1}e^{2x_{2}}\partial_{x_{1}}. \]
Observing that
\[ \det(X_{1},X_{3})(x)=x_{1}e^{x_{2}}, \quad \det(X_{2},X_{3})(x)=x_{1}e^{2x_{2}},\]
\[ \det(X_{1},[X_{1},X_{3}])(x)=e^{2x_{2}},\quad \det([X_{1},X_{3}],[X_{2},X_{3}])(x)=x_{1}e^{3x_{2}},   \]
we obtain $\nu(x)=2$ if $x_1\neq 0$, and $\nu(x)=3$ for $x_1=0$. Thus, the generalized M\'{e}tivier index $$\tilde{\nu}=\max_{x\in \overline{\Omega}}\nu(x)= 3. $$ However, since $\det([X_{1},X_{3}],[X_{2},X_{3}])(x)\neq 0$ for $x_1\neq 0$, the local
homogeneous dimension
\[ Q=\max_{x\in \overline{\Omega}}\left(\max\{d(I)|\lambda_{I}(x)\neq 0\}\right)=4,\]
which clearly indicates that $Q>\tilde{\nu}$ in this degenerate case.
\end{ex}

\begin{remark}
If the H\"{o}rmander vector fields $X$ satisfy M\'{e}tivier's condition on $\overline{\Omega}$ (i.e. $\overline{\Omega}$ is equiregular), then we have $Q=\tilde{\nu}$.
\end{remark}

\subsection{Subunit metric and subunit balls}

Then, we introduce the subunit metric associated with the H\"{o}rmander vector fields $X$.
\begin{definition}[Subunit metric, see \cite{NSW1985,Morbidelli2000}]
\label{def2-1}
For any $x,y\in U$ and $\delta>0$, let $C(x,y,\delta)$ be the collection of absolutely continuous mapping $\varphi:[0,1]\to U$, such that $\varphi(0)=x,\varphi(1)=y$ and
\[ \varphi'(t)=\sum_{i=1}^{m}a_{i}(t)X_{i}I(\varphi(t)) \]
with $\sum_{k=1}^{m}|a_{k}(t)|^2\leq \delta^2$ for a.e. $t\in [0,1]$. Here, $X_{i}I(x)=(b_{i1}(x),b_{i2}(x),\ldots,b_{in}(x))^{T}$ denotes the corresponding vector value function of $X_{i}=\sum_{k=1}^{n}b_{ik}(x)\partial_{x_{k}}$.
The subunit metric $d(x,y)$ is defined by
\begin{equation}\label{2-7}
  d(x,y):=\inf\{\delta>0~|~ \exists \varphi\in C(x,y,\delta)~\mbox{with}~\varphi(0)=x,~ \varphi(1)=y\}.
\end{equation}
\end{definition}

The subunit metric $d$, often referred to as the control distance, is ensured to be well-defined by the Chow-Rashevskii theorem (see \cite[Theorem 57]{Bramanti2014}).   Given any $x\in U$ and $r>0$, we denote by
    \[ B(x,r):=\{y\in U|~d(x,y)<r\} \]
 the subunit ball associated with the subunit metric $d(x,y)$. This notation for the subunit ball will be consistently employed throughout the paper. To provided precise estimates of the volume of the subunit ball, we construct the following lower bound of Nagel-Stein-Wainger polynomial.

\begin{proposition}
\label{prop2-2}
  For any compact subset $K\subset U$ and any $\delta>0$, there exists a positive constant $C>0$ such that
\begin{equation}\label{2-8}
  \Lambda(x,r)\geq Cr^{\tilde{\nu}(K)} \qquad \forall x\in K,~~0<r\leq \delta,
\end{equation}
where  $\tilde{\nu}(K):=\max_{x\in K}\nu(x)$ denotes the generalized M\'{e}tivier index of $K$,  and $\nu(x)$ is the pointwise homogeneous dimension defined in \eqref{2-1}.
\end{proposition}
\begin{proof}
For each fixed $x\in K$, according to \eqref{2-5}, there exists an $n$-tuple $I_x$ such that $d(I_x)=\nu(x)$ and $\lambda_{I_x}(x)\neq 0$. The continuity of the function $y\mapsto|\lambda_{I_x}(y)|$ allows us to select an open neighborhood $U_{x}\subset U$ of $x$ such that $|\lambda_{I_x}(y)|\geq \frac{1}{2}|\lambda_{I_x}(x)|>0$ for all $y\in U_{x}$. Since $K$ is a compact subset of $U$, we can find a finite collection of pairs $(x_{i},I_{i},U_{i}, C_{i})~ (1\leq i\leq q)$ satisfying
\begin{itemize}
  \item $x_{i}\in K$ and $U_{i}\subset U$ is an open neighborhood of $x_i$;
  \item $K\subset \bigcup_{i=1}^{q}U_{i}$;
  \item $d(I_{i})=\nu(x_i)$ and  $|\lambda_{I_i}(y)|\geq C_{i}>0$ for any $y\in U_{i}$.
\end{itemize}
Thus, for any $x\in K$,
\begin{equation}
\begin{aligned}\label{2-9}
\Lambda(x,r)=\sum_{i=1}^{q}\Lambda(x,r)\chi_{U_i\cap K}(x)\geq
\sum_{i=1}^{q}|\lambda_{I_i}(x)|\chi_{U_i\cap K}(x)r^{d(I_i)}\geq C_0\sum_{i=1}^{q}\chi_{U_i\cap K}(x)r^{\nu(x_{i})},
\end{aligned}
\end{equation}
where $C_{0}=\min\{C_{i}|1\leq i\leq q\}$, and $\chi_{E}$ denotes the indicator function of $E$. Observing that $r^{\nu(x_{i})}\geq \left(\min_{1\leq i\leq q}\delta^{\nu(x_i)-\tilde{\nu}(K)}\right) r^{\tilde{\nu}(K)}$ for $0<r\leq \delta$ and $1\leq i\leq q$,
 \eqref{2-9} derives that
\begin{equation*}
 \Lambda(x,r)\geq C r^{\tilde{\nu}(K)} \qquad\forall x\in K,~0<r\leq \delta,
\end{equation*}
where $C=C_0\left(\min_{1\leq i\leq q}\delta^{\nu(x_i)-\tilde{\nu}(K)}\right)>0$.
\end{proof}

Proposition \ref{prop2-2} provides us with the following  volume estimates for the subunit ball, which refine \cite[Theorem 1]{NSW1985}.

\begin{proposition}[Ball-Box theorem]
\label{prop2-3}
For any compact set $K\subset U$, there exist positive constants $0<C_{1}\leq C_{2}$ and $\rho_{K}>0$ such that
\begin{equation}\label{2-10}
 |B(x,r)|\geq C_{1}\Lambda(x,r)\qquad \forall x\in K,~0<r\leq \max\{\delta_{K},\rho_K\},\\
\end{equation}
and
\begin{equation}\label{2-11}
|B(x,r)|\leq C_{2}\Lambda(x,r)\qquad \forall x\in K,~0<r\leq \rho_{K},
\end{equation}
where $|B(x,r)|$ is the $n$-dimensional Lebesgue measure of $B(x,r)$, and $\delta_{K}:=\sup_{x,y\in K}d(x,y)$ denotes the diameter of $K$ with respect to the subunit metric $d$. Furthermore, if the assumption $(A)$ is satisfied (i.e. $|B(x,r)|<+\infty$ for all $x\in U$ and $r>0$), we have
\begin{equation}\label{2-12}
 C_{1}\Lambda(x,r)\leq |B(x,r)|\leq C_{2}\Lambda(x,r)\qquad \forall x\in K,~0<r\leq \max\{\delta_{K},\rho_K\}.\\
\end{equation}
\end{proposition}
\begin{proof}
A well-known result by Nagel-Stein-Wainger \cite{NSW1985} gives that, for any compact set $K\subset U$, there exist positive constants $C>0$ and $\rho_{K}>0$ such that for any $x\in K$ and any $0< r\leq \rho_{K}$,
\begin{equation}\label{2-13}
  C^{-1}\Lambda(x,r)\leq |B(x,r)|\leq C\Lambda(x,r).
\end{equation}
This yields \eqref{2-11}, and also gives \eqref{2-10} and \eqref{2-12} provided $\delta_{K}\leq \rho_{K}$.

 Assume that  $\rho_{K}< \delta_K$ and $\rho_{K} \leq r \leq \delta_K$. Since $d(I)\leq ns_{0}$, we deduce from \eqref{2-3} and \eqref{2-13} that for any $x\in K$ and $\rho_{K} \leq r \leq \delta_K$,
\[\begin{aligned}
|B(x,r)|&\geq |B(x,\rho_{K})|\geq C^{-1}\sum_{I}|\lambda_{I}(x)|\rho_{K}^{d(I)}=C^{-1}\sum_{I}|\lambda_{I}(x)|r^{d(I)}\cdot\left(\frac{\rho_{K}}{r}\right)^{d(I)}\\
&\geq C^{-1}\sum_{I}|\lambda_{I}(x)|r^{d(I)}\cdot\left(\frac{\rho_{K}}{\delta_K}\right)^{d(I)}\geq C^{-1}\left(\frac{\rho_{K}}{\delta_K}\right)^{ns_{0}}\sum_{I}|\lambda_{I}(x)|r^{d(I)}\geq C_{1}\Lambda(x,r),
\end{aligned}
\]
where $C_{1}=C^{-1}\left(\frac{\rho_{K}}{\delta_K}\right)^{ns_{0}}>0$ is a positive constant. This proves \eqref{2-10} as well as the first inequality in \eqref{2-12}.

 Moreover, if the assumption $(A)$ is satisfied, we can choose a point $x_0\in K$ such that
 \[\{x\in U|d(x,K):=\inf_{y\in K}d(x,y)\leq \delta_{K}\}\subset B(x_0,3\delta_K). \]
 Note that $B(x,\delta_K)\subset \{x\in U|d(x,K):=\inf_{y\in K}d(x,y)\leq \delta_{K}\}$ for all $x\in K$.
 Thus, Proposition \ref{prop2-2} implies that for any $x\in K$ and $\rho_{K} \leq r \leq \delta_K$,
\[ \begin{aligned}
|B(x,r)|&\leq|B(x,\delta_K)|\leq |B(x_0,3\delta_K)|\leq C_{2}\Lambda(x,r),
\end{aligned} \]
where $C_{2}>0$ is a positive constant. This completes the proof of \eqref{2-12}.
\end{proof}

By Proposition \ref{prop2-2} and Proposition \ref{prop2-3}, we can deduce that
\begin{proposition}
\label{prop2-4}
 For any compact subset $K\subset U$, there exist  $C_{3}>1$ and $\rho_{K}>0$ such that
\begin{equation}\label{2-14}
  |B(x,2r)|\leq C_{3}|B(x,r)| \quad \forall x\in K,~~ 0<r\leq \frac{\rho_{K}}{2}.
\end{equation}
Additionally,
\begin{equation}\label{2-15}
C_{3}^{-1}|B(y,d(x,y))|\leq   |B(x,d(x,y))|\leq C_{3}|B(y,d(x,y))|\quad \forall x,y\in K~ \mbox{with}~ d(x,y)\leq \frac{\rho_{K}}{2}.
\end{equation}
Here, $\rho_{K}$ is the same positive constant appeared in  Proposition \ref{prop2-3}. Furthermore, if the assumption $(A)$ is satisfied (i.e. $|B(x,r)|<+\infty$ for all $x\in U$ and $r>0$), \eqref{2-14} holds for all $x\in K$ and $0<r\leq \delta_{K}=\sup_{x,y\in K}d(x,y)$, and the restriction $d(x,y)\leq \frac{\rho_{K}}{2}$ in \eqref{2-15} can be removed.
\end{proposition}

Next, we give the following estimates concerning the volume of subunit ball.

 \begin{proposition}
\label{prop2-5}
 Let $W\subset\subset U$ be a bounded open subset, $\mu=\xi+(\eta-1)\tilde{\nu}(\overline{W})$ with $\eta\geq 1$ and $\xi>0$. Then, there exists positive constant $\rho_{\overline{W}}>0$ such that for all $x\in \overline{W}$ and $0< r\leq  \rho_{\overline{W}}$, we have
\begin{equation}\label{2-16}
\int_{B(x,r)}\frac{d(x,y)^\mu}{|B(x,d(x,y))|^{\eta}}dy\leq (C_{1}C)^{1-\eta}C_{3}\frac{2^\xi}{2^{\xi}-1}r^\xi,
\end{equation}
where $\tilde{\nu}(\overline{W})=\max_{x\in\overline{W}}\nu(x)$. Moreover, if  $n\eta>\mu$ and the assumption (A) is satisfied, there exists a positive constant $C>0$ such that
\begin{equation}\label{2-17}
\int_{W} \frac{d(x,y)^\mu}{|B(x,d(x,y))|^{\eta}}dy\leq \left[\left(\frac{C_{2}}{C_{1}}\right)^{\eta}+1\right]\frac{C_{3}}{(CC_{1})^{\frac{\mu}{\tilde{\nu}}}}\frac{2^\xi}{2^{\xi}-1}|W|^\frac{\xi}{\tilde{\nu}(\overline{W})}\qquad \forall x\in \overline{W}.
\end{equation}
Here, $C,C_{1},C_{2},C_{3}$ are the positive constants appeared in Proposition \ref{prop2-2}-Proposition \ref{prop2-4}.
 \end{proposition}
\begin{proof}
For any $x\in \overline{W}$, $y\in B(x,r)$ and $0<r\leq \rho_{\overline{W}}$, we have $d(x,y)<\rho_{\overline{W}}$. Proposition \ref{prop2-2} and \eqref{2-10} imply that
  $d(x,y)^\mu\leq (C_{1}C)^{1-\eta} d(x,y)^\xi|B(x,d(x,y))|^{\eta-1}$. Then, using \eqref{2-14} we get
  \begin{equation*}
    \begin{aligned}
   \int_{B(x,r)}\frac{d(x,y)^\mu}{|B(x,d(x,y))|^{\eta}}dy&\leq    (C_{1}C)^{1-\eta}\int_{B(x,r)}\frac{d(x,y)^\xi}{|B(x,d(x,y))|}dy\\
    &= (C_{1}C)^{1-\eta}\sum_{k=0}^{\infty}\int_{\{y\in U|\frac{r}{2^{k+1}}\leq d(x,y)<\frac{r}{2^{k}}\}}\frac{d(x,y)^\xi}{|B(x,d(x,y))|}dy\\
    &\leq (C_{1}C)^{1-\eta}\sum_{k=0}^{\infty}\int_{\{y\in U|\frac{r}{2^{k+1}}\leq d(x,y)< \frac{r}{2^{k}}\}}\frac{1}{|B\left(x,\frac{r}{2^{k+1}}\right)|}\left( \frac{r}{2^{k}}\right)^\xi dy\\
    &\leq (C_{1}C)^{1-\eta}\sum_{k=0}^{\infty}\frac{|B\left(x,\frac{r}{2^{k}}\right)|}{|B\left(x,\frac{r}{2^{k+1}}\right)|}\left( \frac{r}{2^{k}}\right)^\xi \leq (C_{1}C)^{1-\eta}C_{3}\sum_{k=0}^{\infty}\left( \frac{r}{2^{k}}\right)^\xi\\
    &= (C_{1}C)^{1-\eta}C_{3}\frac{2^\xi}{2^{\xi}-1}r^\xi,
    \end{aligned}
  \end{equation*}
 which yields \eqref{2-16}. Additionally, if the assumption (A) is satisfied, we can derive from Proposition \ref{prop2-4} that \eqref{2-16} holds for all $x\in \overline{W}$ and $0<r\leq \max\{\rho_{\overline{W}},\delta_{\overline{W}}\}$, where $\delta_{\overline{W}}=\sup_{x,y\in \overline{W}}d(x,y)$ denotes the diameter of $\overline{W}$ with respect to the subunit metric $d$.

Next, we construct the estimate \eqref{2-17} under assumption (A) and $n\eta>\mu$. For any $x\in \overline{W}$,  we can choose a suitable radius $r_0>0$ such that $|B(x,r_0)|=|W|$. Then, we have
\begin{equation}
\begin{aligned}\label{2-18}
   |W\setminus B(x,r_0)|&=|W\setminus (W\cap B(x,r_0))|=|W|-|W\cap B(x,r_0)|\\
   &=|B(x,r_0)|-|W\cap B(x,r_0)|=|B(x,r_0)\setminus W|.
   \end{aligned}
\end{equation}
 We claim that $r_{0}\leq   \delta_{\overline{W}}=\sup_{x,y\in \overline{W}}d(x,y)$. Indeed,   if $r_{0}> \delta_{\overline{W}}$, then for any $y\in \overline{W}$, we have $d(x,y)\leq \delta_{\overline{W}}<r_{0}$, which implies that $\overline{W}\subset B(x,r_0)$ and $|W|<|B(x,r_0)|$. This contradicts the fact that $|B(x,r_0)|=|W|$.

If $n\eta>\mu$, by \eqref{2-3} we see that $t\mapsto \frac{t^\mu}{\Lambda(x,t)^\eta} $ is a decreasing function on $\mathbb{R}^{+}$. It follows from Proposition \ref{prop2-3} and \eqref{2-18} that, for any $x\in \overline{W}$,
\begin{equation}\label{2-19}
\begin{aligned}
  \int_{W\setminus B(x,r_0)}\frac{d(x,y)^\mu}{|B(x,d(x,y))|^{\eta}}dy&\leq C_{1}^{-\eta}\int_{W\setminus B(x,r_0)}\frac{d(x,y)^\mu}{\Lambda(x,d(x,y))^{\eta}}dy\leq C_{1}^{-\eta}|W\setminus B(x,r_0)| \frac{r_{0}^\mu}{\Lambda(x,r_{0})^\eta}\\
  &=C_{1}^{-\eta}|B(x,r_0)\setminus W| \frac{r_{0}^\mu}{\Lambda(x,r_{0})^\eta}
  \leq C_{1}^{-\eta}\int_{B(x,r_0)\setminus W}\frac{d(x,y)^\mu}{\Lambda(x,d(x,y))^{\eta}}dy\\
  &  \leq \left(\frac{C_{2}}{C_{1}}\right)^{\eta}\int_{B(x,r_0)\setminus W}\frac{d(x,y)^\mu}{|B(x,d(x,y))|^{\eta}}dy.
  \end{aligned}
\end{equation}
 Therefore, we conclude from  Proposition \ref{prop2-2} and Proposition \ref{prop2-3}, \eqref{2-16}, and \eqref{2-19} that
\[ \begin{aligned}
\int_{W} \frac{d(x,y)^\mu}{|B(x,d(x,y))|^{\eta}}dy &=\int_{W\setminus B(x,r_0)}\frac{d(x,y)^\mu}{|B(x,d(x,y))|^{\eta}}dy+\int_{W\cap B(x,r_0)} \frac{d(x,y)^\mu}{|B(x,d(x,y))|^{\eta}}dy\\
&\leq \left[\left(\frac{C_{2}}{C_{1}}\right)^{\eta}+1\right]\int_{B(x,r_0)}\frac{d(x,y)^\mu}{|B(x,d(x,y))|^{\eta}}dy\\
&\leq \left[\left(\frac{C_{2}}{C_{1}}\right)^{\eta}+1\right]\frac{C_{3}}{(CC_{1})^{\eta-1}}\frac{2^\xi}{2^{\xi}-1}r_{0}^\xi\\
&\leq \left[\left(\frac{C_{2}}{C_{1}}\right)^{\eta}+1\right]\frac{C_{3}}{(CC_{1})^{\frac{\mu}{\tilde{\nu}}}}\frac{2^\xi}{2^{\xi}-1}|W|^\frac{\xi}{\tilde{\nu}(\overline{W})},
\end{aligned}\]
which gives \eqref{2-17}.
\end{proof}

\subsection{Degenerate Friedrichs-Poincar\'{e} type inequality}

For H\"ormander vector fields, the Poincar\'{e}-Wirtinger type inequality has attracted considerable attention in the literature since Jerison's work \cite{Jerison1986duke}. Further investigations have been carried out by Saloff-Coste \cite{Saloff-Coste1992}, Garofalo-Nhieu \cite{Garofalo1996}, and Haj{\l}asz-Koskela \cite{Hajlasz2000}, among others.  However, when addressing the Dirichlet problems of degenerate elliptic equations, we require the following Friedrichs-Poincar\'{e} type inequality, which is entirely different from the Poincar\'{e}-Wirtinger type inequality and much less known.

 \begin{proposition}[Degenerate Friedrichs-Poincar\'{e} Inequality]
\label{prop2-6}
Let $X=(X_{1},X_{2},\ldots,X_{m})$ be the smooth vector fields defined on $U$, satisfying the
H\"{o}rmander's condition $(H)$. For any open bounded subset $W\subset\subset U$ and positive number $p\geq 1$, there exists a positive constant $C>0$ such that
\begin{equation}\label{2-20}
 \int_{W}{|u|^pdx}\leq C\int_{W}|Xu|^pdx,\qquad \forall u\in \mathcal{W}_{X,0}^{1,p}(W).
  \end{equation}
\end{proposition}

 The statement of \eqref{2-20} originated from \cite[Lemma 5]{Xu90} assumes the smoothness of the boundary $\partial W$   and the existence of at least one vector field $X_{j}~ (1\leq j\leq m)$ that can be globally straightened in $W$. Moreover, for the case when $p=2$, \eqref{2-20} was also discussed in \cite[Lemma 3.2]{jost1998} and \cite[Proposition 2.1]{chen-chen2020}, provided there exists an additional non-characteristic condition on the smooth boundary of $W$.
  However, it's worth noting that for general H\"{o}rmander vector fields, the characteristic set in the boundary may not be empty, even for a smooth domain (e.g., a unit ball in the Heisenberg group). Therefore, the smoothness and non-characteristic assumptions might be too restrictive in many degenerate situations. To eliminate these limitations, we propose a generalization of \eqref{2-20} to the arbitrary open bounded subsets $U_1$, without imposing any additional assumptions on its boundary. This generalization extends the applicability of the Friedrichs-Poincar\'{e} inequality encompass a wider range of scenarios.

\begin{proof}[Proof of Proposition \ref{prop2-6}]
We prove \eqref{2-20} by contradiction. Suppose that
\[ \inf_{\|\varphi\|_{L^{p}(W)}=1,~~ \varphi\in \mathcal{W}_{X,0}^{1,p}(W)}\|X\varphi\|^p_{L^p(W)}=0.\]
 Then there is a sequence
$\{\varphi_{j}\}_{j=1}^{\infty}$ in  $\mathcal{W}_{X,0}^{1,p}(W)$ such that
$\|X\varphi_{j}\|_{L^p(W)}\to 0$ with
$\|\varphi_{j}\|_{L^p(W)}=1$.   \cite[Corollary 3.3]{Danielli1991} indicates that $\mathcal{W}_{X,0}^{1,p}(W)$ is compactly embedded into $L^p(W)$ for $p\geq 1$. This allows us to select a subsequence $\{\varphi_{j_{i}}\}_{i=1}^{\infty}\subset \{\varphi_{j}\}_{j=1}^{\infty}\subset \mathcal{W}_{X,0}^{1,p}(W)$ such that $\varphi_{j_{i}}\to \varphi_{0}$ in $L^{p}(W)$ and $\varphi_{0}\in L^{p}(W)$ with $\|\varphi_0\|_{L^{p}(W)}=1$. Now, for any $1\leq l\leq m$ and $u\in C_{0}^{\infty}(W)$, we have
\begin{equation}\label{2-21}
  \int_{W}\varphi_{0}X_{l}^{*}udx=\lim_{i\to \infty}\int_{W}\varphi_{j_{i}}X_{l}^{*}udx=\lim_{i\to \infty}\int_{W}uX_{l}\varphi_{j_{i}}dx=0.
\end{equation}
Hence, $X\varphi_0=0$, $\varphi_{j_{i}}\to \varphi_{0}$ in $\mathcal{W}_{X,0}^{1,p}(W)$,  and $\varphi_0\in \mathcal{W}_{X,0}^{1,p}(W)$. Denote by $\triangle_{X}:=-\sum_{j=1}^{m}X_{j}^{*}X_{j}$ the formal self-adjoint H\"{o}rmander operator associated with $X=(X_{1},X_{2},\ldots,X_m)$.  Substituting $u$ with $X_{l}u$
in \eqref{2-21} yields:
\begin{equation*}
 (\varphi_{0},\triangle_{X}u)_{L^{2}(W)}=-\sum_{l=1}^{m}\int_{W}\varphi_{0}X_{l}^{*}X_{l}udx=0,
\end{equation*}
which implies $\triangle_{X}\varphi_0=0$ in $\mathcal{D}'(W)$. The hypoellipticity of  $\triangle_{X}$ yields that $\varphi_{0}\in C^{\infty}(W)$. Moreover, since $X_{j}\varphi_0=0$ on $W$ for $1\leq j\leq m$ and $\|\varphi_{0}\|_{L^p(W)}=1$, the  H\"{o}rmander's condition implies that $\partial_{x_{j}}\varphi_0=0$ on $W$ for $1\leq j\leq n$. That means  $\varphi_0$ must be a non-zero constant on $W$.

Next, we choose the sequence $\{u_k\}_{k=1}^{\infty}\subset C_{0}^{\infty}(W)$ such that $u_{k}\to \varphi_{0}$ in $\mathcal{W}_{X,0}^{1,p}(W)$, and we denote by
\[ \overline{u_{k}}:=\left\{
                     \begin{array}{ll}
                      u_{k}, & \hbox{on $W$,} \\
                       0, & \hbox{on $U\setminus W$;}
                     \end{array}
                   \right.\qquad \mbox{and}\qquad\overline{\varphi_{0}}:=\left\{
                     \begin{array}{ll}
                       \varphi_{0}, & \hbox{on $W$,} \\
                       0, & \hbox{on $U\setminus W$.}
                     \end{array}
                   \right.\]
It follows that $\{\overline{u_{k}}\}_{k=1}^{\infty}\subset C_{0}^{\infty}(U)$ is a Cauchy sequence in $\mathcal{W}_{X,0}^{1,p}(U)$ with $\overline{u_{k}}\to \overline{\varphi_{0}}$ in $L^p(U)$. As a result, we have
$\overline{\varphi_{0}}\in \mathcal{W}_{X,0}^{1,p}(U)$,
$ \|\overline{u_{k}}-\overline{\varphi_{0}}\|_{\mathcal{W}_{X,0}^{1,p}(U)}=\|u_{k}-\varphi_{0}\|_{\mathcal{W}_{X,0}^{1,p}(W)}\to 0$ and
\[ \int_{U}|X\overline{\varphi_{0}}|^{p}dx=\int_{W}|X\varphi_{0}|^{p}dx=0.\]
Let $q=\frac{p}{p-1}$ for $p>1$ and $q=\infty$ for $p=1$. For any $v\in C_{0}^{\infty}(U)$, we have
\begin{equation*}
  |(\overline{\varphi_0},\triangle_{X}v)_{L^2(U)}|\leq \sum_{j=1}^{m}|(X_{j}\overline{\varphi_0},X_{j}v)_{L^2(U)}|\leq \sum_{j=1}^{m}\|X_{j}\overline{\varphi_0}\|_{L^{p}(U)}\cdot \|X_{j}v\|_{L^{q}(U)}=0,
\end{equation*}
and therefore $\triangle_{X}\overline{\varphi_0}=0$ in $\mathcal{D}'(U)$. The hypoellipticity of  $\triangle_{X}$ also gives $\overline{\varphi_{0}}\in C^{\infty}(U)$, which leads a contradiction since $\overline{\varphi_{0}}$ is not smooth across $\partial W$.
\end{proof}

\subsection{Chain rules in generalized Sobolev spaces}

Before delving into the chain rules in the generalized Sobolev spaces, we  give the following useful
 proposition.
 \begin{proposition}
\label{prop2-7}
Let  $\Omega\subset\subset U$ be a bounded open subset. For $p\geq 1$ and any $u\in \mathcal{W}_{X}^{1,p}(U)$, if $\text{\rm supp}~u$ is a compact subset of $\Omega$, then $u\in \mathcal{W}_{X,0}^{1,p}(\Omega)$.
\end{proposition}

\begin{proof}
Clearly, $u\in \mathcal{W}_{X}^{1,p}(U)$ implies $u\in \mathcal{W}_{X}^{1,p}(\Omega)$. Since $\text{supp}~u$ is a compact subset in $\Omega$, there exists a function $f\in C_{0}^{\infty}(\Omega)$ such that $f\equiv 1$ on $\text{supp}~u$ and $\text{supp}~f\subset  \Omega $. Owing to the Meyers-Serrin's theorem (see \cite[Theorem 1.13]{Garofalo1996}), we can find  a sequence $\{\psi_{i}\}_{i=1}^{\infty}\subset C^{\infty}(\Omega)\cap \mathcal{W}_{X}^{1,p}(\Omega)$ such that $\psi_{i}\to u$ in $\mathcal{W}_{X}^{1,p}(\Omega)$. Observing that $f\psi_{i}\in C_{0}^{\infty}(\Omega)$ and
\[ \begin{aligned}
\|f\psi_i-u\|_{\mathcal{W}_{X}^{1,p}(\Omega)}^{p}&=\|f\psi_i-fu\|_{\mathcal{W}_{X}^{1,p}(\Omega)}^{p}= \|X(f(\psi_i-u))\|_{L^p( \Omega )}^{p}+\|f(\psi_i-u)\|_{L^p( \Omega )}^{p}\\
&\leq C(\|X(\psi_i-u)\|_{L^p(\Omega)}^p+\|\psi_i-u\|_{L^p(\Omega)}^p)\to 0
\end{aligned}
\]
as $i\to \infty$, we conclude that $u\in \mathcal{W}_{X,0}^{1,p}( \Omega)$.
\end{proof}

Then, we have

\begin{proposition}[Chain rules]
\label{prop2-8}
Let  $U_{1}$ be an open subset of $U$. Suppose that $F\in C^{1}(\mathbb{R})$ with $F'\in L^{\infty}(\mathbb{R})$. Then for any $u\in \mathcal{W}^{1,p}_X(U_{1})$ with $p\geq 1$, we have
\begin{equation}\label{2-22}
  X_{j}(F(u))=F'(u)X_{j}u\qquad \mbox{in}~~\mathcal{D}'(U_{1})\qquad \mbox{for}~~j=1,\ldots,m.
\end{equation}
Moreover,
\begin{enumerate}[(1)]
\item if $F(0)=0$, then $F(u)\in \mathcal{W}^{1,p}_X(U_{1})$;
  \item if $F(0)=0$ and $u\in \mathcal{W}_{X,0}^{1,p}(U_{1})$, then $F(u)\in \mathcal{W}^{1,p}_{X,0}(\Omega)$.
\end{enumerate}
\end{proposition}

\begin{proof}
The Meyers-Serrin theorem (see \cite[Theoerm 1.13]{Garofalo1996}) tells us that for any $u\in \mathcal{W}_{X}^{1,p}(U_{1})$, there exists $\{u_ {k}\}_{k=1}^{\infty}\subset C^{\infty}(U_{1})\cap \mathcal{W}_{X}^{1,p}(U_{1})$ such that $u_{k}\to u$ in $\mathcal{W}_{X}^{1,p}(U_{1})$. Since for any $x\in \mathbb{R}$
  \[|F(x)|\leq |F(x)-F(0)|+|F(0)|\leq \|F'\|_{L^{\infty}(\mathbb{R})} |x|+|F(0)|,\]
it follows that $F(u),F(u_k)\in L^p_{\rm loc}(U_{1})\subset \mathcal{D}'(U_{1})$. Additionally,
  \begin{equation}\label{2-23}
    \int_{U_{1}}|F(u_k)-F(u)|^pdx\leq \|F'\|_{L^{\infty}(\mathbb{R})}^p\int_{U_{1}}|u_k-u|^pdx\to0,
  \end{equation}
which implies $F(u_k)\to F(u)$ in $L^p(U_{1})$.  Recalling that $u_k\to u$ in $L^p(U_{1})$,  there is a subsequence $\{u_{k_{i}}\}_{i=1}^{\infty}$ that converges to $u$ almost everywhere on $U_{1}$. Thus,  $F'(u_{k_{i}})\to F'(u)$ almost everywhere on $U_{1}$. For each $1\leq j\leq m$, applying the dominated convergence theorem, we obtain
\begin{equation}
  \begin{aligned}\label{2-24}
   &\int_{U_{1}}|X_j (F(u_{k_{i}}))-F'(u)X_ju|^pdx=\int_{U_{1}}|F'(u_{k_{i}})X_ju_{k_{i}}-F'(u)X_ju|^pdx\\
    &\leq C\int_{U_{1}}|F'(u_{k_{i}})X_{j}u_{k_{i}}-F'(u_{k_{i}})X_{j}u|^pdx
    +C\int_{U_{1}}|F'(u_{k_{i}})X_{j}u-F'(u)X_{j}u|^pdx \\
     & \leq C \|F'\|_{L^{\infty}(\mathbb{R})}^p\int_{U_{1}}|X_ju_{k_{i}}-X_ju|^pdx+C\int_{U_{1}}|F'(u_{k_{i}})-F'(u)|^p|X_ju|^pdx\to 0,
  \end{aligned}
  \end{equation}
  which implies that $X_j(F(u_{k_{i}}))=F'(u_{k_{i}})X_{j}u_{k_{i}}\in L^p(U_{1})$ converges to  $F'(u)X_ju$ in $L^p(U_{1})$. Therefore, $F'(u_{k_{i}})X_{j}u_{k_{i}}\to F'(u)X_ju$ in $\mathcal{D}'(U_1)$.  Observing that for any $\varphi\in C_0^\infty(U_{1})$, we have
\begin{equation}\label{2-25}
\int_{U_{1}}F'(u_{k_{i}})X_ju_{k_{i}}\varphi dx=\int_{U_{1}}X_j(F(u_{k_{i}}))\varphi dx=\int_{U_{1}}F(u_{k_{i}})X_j^*\varphi dx.
\end{equation}
Letting $i\to \infty$ in \eqref{2-25} and using $F(u_{k_{i}})\to F(u)$  in $\mathcal{D'}(U_{1})$, we obtain
\[ \int_{U_{1}}F'(u)X_ju\varphi dx=\int_{U_{1}}F(u)X_j^*\varphi dx\qquad \forall \varphi\in C_0^\infty(U_{1}).\]
  Hence, \eqref{2-22} is proved. \par

  If $F(0)=0$, then $F(u)\in L^p(U_{1})$. It derives from \eqref{2-22} that
 \[\|XF(u)\|_{L^p(U_1)}\leq \|F'\|_{L^{\infty}(\mathbb{R}^n)}\|Xu\|_{L^p(U_1)}<\infty.\] Thus, $F(u)\in \mathcal{W}_{X}^{1,p}(U_{1})$. Suppose further that $u\in \mathcal{W}_{X,0}^{1,p}(U_{1})$, then we can choose an approximating sequence $\{u_k\}_{k=1}^{\infty}\subset C_0^\infty(U_{1})$ such that $\|u_k-u\|_{\mathcal{W}^{1,p}_X(U_{1})}\to 0$ and $F(u_k)\in C_0^1(U_{1})$. It follows from \eqref{2-23} and \eqref{2-24}  that $\|F(u_{k_{i}})-F(u)\|_{\mathcal{W}^{1,p}_X(U_{1})}\to 0$, which yields $F(u)\in \mathcal{W}_{X,0}^{1,p}(U_{1})$.
\end{proof}
As a result of Proposition \ref{prop2-8}, we have
\begin{proposition}
\label{prop2-9}
Let  $U_{1}$ be an open subset of $U$. For any $u\in \mathcal{W}_X^{1,p}(U_{1})$ and any $c\in \mathbb{R}$, we have
\begin{equation}\label{2-26}
  X_j(u-c)_+=H(u-c)X_ju\quad\mbox{and}\quad X_j(u-c)_-=-H(c-u)X_ju~~~\mbox{in}~~\mathcal{D}'(U_{1}),
\end{equation}
where $H(x)=\chi_{\{x\in \mathbb{R}|x>0\}}(x)$ and $\chi_{E}$ denotes the indicator function of $E$. Furthermore,
\begin{enumerate}[(1)]
  \item if $c\geq 0$, we have $(u-c)_{+}, (u+c)_{-}\in \mathcal{W}_X^{1,p}(U_{1})$;
  \item if $c\geq 0$ and $u\in \mathcal{W}_{X,0}^{1,p}(U_{1})$, then $(u-c)_{+}, (u+c)_{-}\in \mathcal{W}^{1,p}_{X,0}(U_{1})$.
\end{enumerate}
\end{proposition}
\begin{proof}
Suppose first that $c\in \mathbb{R}$. For any $\varepsilon>0$, we define
\begin{equation*}
F_{c,\varepsilon}(x)=\left\{
            \begin{array}{ll}
             [(x-c)^2+\varepsilon^2]^{\frac{1}{2}}-\varepsilon, & \hbox{$x>c$;} \\[2mm]
               0, & \hbox{$x\leq c$.}
            \end{array}
          \right.
\end{equation*}
It follows that $F_{c,\varepsilon}\in C^1(\mathbb{R})$, $|F_{c,\varepsilon}'(x)|\leq 1$ and
\begin{equation*}
 F_{c,\varepsilon}'(x)=\left\{
            \begin{array}{ll}
\frac{x-c}{\sqrt{(x-c)^2+\varepsilon^2}}, & \hbox{$x>c$;} \\[2mm]
               0, & \hbox{$x\leq c$.}
            \end{array}
          \right.
\end{equation*}
Moreover, we have $\lim_{\varepsilon\to 0}F_{c,\varepsilon}(x)=(x-c)_{+}$ and $\lim_{\varepsilon\to 0}F'_{c,\varepsilon}(x)=\chi_{\{x|x>c\}}(x)$ for all $x\in \mathbb{R}$. For any $u\in \mathcal{W}_{X}^{1,p}(U_{1})$ and $j=1,\ldots, m$, Proposition \ref{prop2-8} yields that
 \begin{equation}\label{2-27}
  \int_{U_{1}}F_{
c,\varepsilon}(u)X_{j}^{*}\varphi dx=\int_{U_{1}}\varphi F_{c,\varepsilon}'(u)X_{j}u dx\qquad \forall \varphi\in C_{0}^{\infty}(U_{1}).
 \end{equation}
Letting $\varepsilon\to 0^{+}$ in \eqref{2-27}, by dominated convergence theorem we obtain
\[ \int_{U_{1}}(u-c)_{+}X_{j}^{*}\varphi dx=\int_{\{x\in U_{1}| u(x)>c\}}\varphi X_{j}u dx, \qquad\forall \varphi\in C_{0}^{\infty}(U_{1}). \]
Therefore, $X_j(u-c)_+=H(u-c)X_ju$ in $\mathcal{D}'(U_{1})$. Similarly, we can deduce that $X_j(u-c)_-=-H(c-u)X_ju$ in $\mathcal{D}'(U_{1})$ by replacing $F_{c,\varepsilon}(x)$ to $G_{c,\varepsilon}(x):=F_{-c,\varepsilon}(-x)$ in the arguments above. Thus, \eqref{2-26} is achieved.

Assuming that $c\geq 0$, by $(x-c)_{+}\leq |x|$ and \eqref{2-26} we obtain $(u-c)_{+}\in \mathcal{W}_{X}^{1,p}(U_{1})$. If we further suppose that $u\in \mathcal{W}_{X,0}^{1,p}(U_{1})$, Proposition \ref{prop2-8} gives that $F_{c,\varepsilon}(u)\in \mathcal{W}_{X,0}^{1,p}(U_{1})$. Note that
\begin{equation*}
  \int_{U_{1}}|F_{c,\varepsilon}(u)-(u-c)_{+}|^{p}dx\leq \int_{U_{1}}|(u-c)_{+}|^{p}dx\leq \int_{U_{1}}|u|^p dx
\end{equation*}
and
\begin{equation*}
 \int_{U_{1}}|X_{j}F_{c,\varepsilon}(u)-X_j(u-c)_{+}|^{p}dx=\int_{U_{1}}|F'_{c,\varepsilon}(u)-H(u-c)|^{p}|X_ju|^{p}dx\leq 2^{p}\int_{U_{1}}|X_ju|^{p}dx.
\end{equation*}
Using the dominated convergence theorem we obtain $F_{c,\varepsilon}(u)\to (u-c)_{+}$ in $\mathcal{W}_{X,0}^{1,p}(U_{1})$, and therefore $(u-c)_{+}\in \mathcal{W}_{X,0}^{1,p}(U_{1})$.  The proof for the situation $(u+c)_{-}$ is similar and we omit here.
\end{proof}

\section{Representation formulas and weak-$L^p$ estimates of $T$ operators}
\label{section3}

In this section, we provide the key tools for constructing Sobolev inequalities on $\mathcal{W}_{X,0}^{k,p}(\Omega)$, including two different types of representation formulas and weighted weak-$L^p$ estimates.
\subsection{Two types of Representation formulas}

\begin{proposition}[Type I representation formula, see {\cite[Proposition 2.12]{Franchi1995}} and \cite{Lu-Wheeden1998}]
\label{prop3-1}
Suppose $X=(X_{1},X_{2},\ldots,X_{m})$ satisfy the condition (H), and $W\subset \subset U$ is a bounded open subset. Then, there exist positive constants $C>0$, $\alpha_{1}\geq 1$ and $r_0>0$ such that for any $x_0\in \overline{W}$ and $0<r<r_0$, we have
\begin{equation}\label{3-1}
  |f(x)-f_{B}|\leq C\int_{B(x_{0},\alpha_{1}r)}\frac{d(x,y)}{|B(x,d(x,y))|}|Xf(y)|dy
\end{equation}
holds for all $x\in B(x_0,r)$ and any $f\in C^{\infty}(B(x_0,\alpha_{1}r))$. Here $\alpha_{1}$ is a positive constant independent of $f$ and $B(x_0,r)$, and $f_{B}=|B(x_0,r)|^{-1}\int_{B(x_0,r)}f(y)dy$.
\end{proposition}

The original proof of \eqref{3-1} in \cite[Proposition 2.12]{Franchi1995} consists of an elaborate argument relying directly on the lifting technique introduced by Rothschild-Stein \cite{stein1976}, which also improved the previous results in \cite[pp. 384-388]{Lu1992}. As is well-known,  the type I representation formula \eqref{3-1} plays a crucial role in establishing the Poincar\'{e}-Wirtinger type inequality and subelliptic Sobolev embedding in the space $\mathcal{W}_{X}^{1,p}(\Omega)$ associated with certain special domains (e.g. X-PS domain), as well as the relative isoperimetric inequality (see \cite{Garofalo1996, Lu1992, Franchi-Lu-Wheeden1996}). In Section \ref{section4} below, this representation formula will be utilized in proving the Sobolev inequality in the supercritical case $kp>\tilde{\nu}$.

We must mention that, to construct the Sobolev inequalities on $\mathcal{W}_{X,0}^{k,p}(\Omega)$ in the subcritical case $kp<\tilde{\nu}$ and critical case $kp=\tilde{\nu}$, a different type of representation formula is required, which is entirely distinct from \eqref{3-1}.

\begin{proposition}[Type II representation formula]
\label{prop3-2}
Assume $X=(X_{1},X_{2},\ldots,X_{m})$ satisfy the condition (H), and $W\subset \subset U$ is a bounded open subset. Then, there exists a positive constant $C>0$ such that for any $f\in C_{0}^{\infty}(W)$, we have
\begin{equation}\label{3-2}
  |f(x)|\leq C\int_{W}\frac{d(x,y)}{|B(x,d(x,y))|}(|Xf(y)|+|f(y)|)dy, \qquad \forall x\in \overline{W}.
\end{equation}
\end{proposition}

The type II representation formula \eqref{3-2} without the term $f$ in the right-hand side has been previously presented in  \cite[Proposition 2.4]{Capogna1993}, \cite[p. 210, Remark]{Capogna1994} and \cite[Lemma 5.1]{Lu1992}. As indicated in those references, it was obtained from the identity
\begin{equation}\label{3-3}
  u(x)=\int_{W}X\Gamma(x,y)\cdot Xu(y)dy
\end{equation}
and the associated estimate
\begin{equation}\label{3-4}
  |X\Gamma(x,y)|\leq C\frac{d(x,y)}{|B(x,d(x,y))|}
\end{equation}
mentioned in \cite[p. 114, Corollary]{NSW1985}, where $\Gamma(x,y)$ denotes the global fundamental solution of the H\"{o}rmander operator $-\triangle_{X}=\sum_{j=1}^{m}X_{j}^{*}X_{j}$ in $\mathbb{R}^n$. However, as pointed out by Nagel \cite[Theorem 11]{Nagel1986} and  Biagi-Bonfiglioli-Bramanti \cite[p. 1882]{Biagi2022}, the kernel function $\Gamma(x,y)$ in \eqref{3-3} and \eqref{3-4}, derived by locally saturating the lifted variables for the parametrix $\widetilde{\Gamma}$ associated with the lifting operators $-\widetilde{\triangle_{X}}=\sum_{j=1}^{m}\widetilde{X_{j}^{*}}\widetilde{X_{j}}$, should only be expected to be a local parametrix for $-\triangle_{X}$, but not a genuine global fundamental solution in $\mathbb{R}^n$. Therefore, the identity \eqref{3-3} is invalid, and in our opinion, there exist gaps in the proofs provided in the aforementioned references.

To provide a rigorous proof of Proposition \ref{prop3-2}, we invoke the celebrated lifting-approximating theory by Rothschild-Stein  \cite[Part II]{stein1976}. In precisely, we have

\begin{proposition}[{\cite[Theorem 10.6, Theorem 10.7, Proposition 10.39, Corollary 10.37]{Bramanti2022}}]
\label{prop3-3}
For any $x_{0}\in U$, suppose that the vector fields $X=(X_{1},X_{2},\ldots,X_{m})$ satisfying H\"{o}rmander's condition of step $s(x_0)$ at $x_{0}$. Then we have:
\begin{enumerate}[(1)]
  \item There exists an integer $l$ and vector fields $\widetilde{X_{1}}, \widetilde{X_{2}},\ldots,\widetilde{X_{m}}$ defined in an open neighborhood $W_{0}$ of $(x_{0},0)\in \mathbb{R}^{n+l}$, of the form
\[ \widetilde{X_{j}}=X_{j}+\sum_{i=1}^{l}b_{ji}(x,t_{1},t_{2},\ldots,t_{i-1})\partial_{t_{i}}, \]
where $b_{ji}$ are polynomials such that the vector fields $\widetilde{X_{1}}, \widetilde{X_{2}},\ldots,\widetilde{X_{m}}$ are free up to step $s(x_0)$ and satisfying H\"{o}rmander's condition of step $s(x_0)$ in  $W_{0}$.
\item There exists a structure of stratified homogeneous group $\mathbb{G}$ in $\mathbb{R}^{N}=\mathbb{R}^{n+l}$,  with canonical generators $Y_{1},Y_{2},\ldots,Y_{m}$, and for any $\eta$ in a neighborhood of $(x_0,0)$, there exists a smooth diffeomorphism $\xi\mapsto \Theta_{\eta}(\xi)$ from a neighborhood of $\eta$ onto a neighborhood of the origin in $\mathbb{G}$, smoothly depending on $\eta$, such that for any smooth function $f:\mathbb{G}\to \mathbb{R}$,
\[ \widetilde{X_{j}}(f(\Theta_{\eta}(\cdot)))(\xi)=(Y_{j}f+R_{\eta,j}f)(\Theta_{\eta}(\xi)),\qquad j=1,\ldots,m. \]
Here, the remainder $R_{\eta,j}$, depends smoothly on $\eta$, is a smooth vector field with local degree $\leq 0$.
\item The function $\Theta(\eta,\xi):=\Theta_{\eta}(\xi)$ satisfies $\Theta(\eta,\xi)^{-1}=-\Theta(\xi,\eta)$. Moreover, the change of coordinates from $\mathbb{R}^{N}$ to $\mathbb{G}$ is given by $\xi\to u=\Theta(\eta,\xi)$, which admits a Jacobian determinant such that $d\xi=c(\eta)(1+O(\|u\|_{\mathbb{G}}))du$. Here, $c(\eta)$ is a smooth function, bounded and bounded away from zero, and $\|\cdot\|_{\mathbb{G}}$ is a homogeneous norm on $\mathbb{G}$.

\item Let $\tilde{d}$ be the subunit metric induced by the vector fields $\widetilde{X_{1}}, \widetilde{X_{2}},\ldots,\widetilde{X_{m}}$ on $W_{0}$. For any homogeneous norm $\|\cdot\|_{\mathbb{G}}$ on $\mathbb{G}$, the function $\rho(\xi,\eta)=\|\Theta(\xi,\eta)\|_{\mathbb{G}}$ is a quasidistance, locally equivalent
to the subunit metric $\tilde{d}$.

\item For any points $\xi=(x,s)$ and $\eta=(y,t)$ in $W_{0}$, we have
\begin{equation}\label{3-5}
 \tilde{d}(\xi,\eta)=\tilde{d}((x,s),(y,t))\geq d(x,y).
\end{equation}

\item Let $Q_{0}$ be the homogeneous dimension of $\mathbb{G}$, then for any compact subset $K\subset W_0$, there exist $C>1$ and $\delta_{0}>0$ such that
\begin{equation}\label{3-6}
 C^{-1}r^{Q_{0}}\leq |\tilde{B}(\xi,r)|\leq  Cr^{Q_{0}}
\end{equation}
holds for $\xi\in K$ and $0<r\leq\delta_{0}$. Here, $\tilde{B}(\xi,r)=\{\eta\in W_{0}|\tilde{d}(\xi,\eta)<r\}$ denotes the subunit ball induced by the subunit metric $\tilde{d}$.

\item For any $\xi=(x,s)\in\mathbb{R}^{n+l}$, we denote by $\pi_{1}(\xi):=\pi_{1}(x,s)=x$ the projection from $\mathbb{R}^{n+l}$ to $\mathbb{R}^{n}$. Then,  $\pi_{1}(\tilde{B}(\xi,r))=B(x,r)$.

\item For any compact set $K\subset W_0$, there exist positive constants $C>0$ and $\delta_0>0$ such that
\begin{equation}\label{3-7}
  |\{t\in \mathbb{R}^{l}|(y,t)\in \tilde{B}((x,0),r)\}|\leq C\frac{|\tilde{B}((x,0),r)|}{|B(x,r)|}
\end{equation}
holds for any $(x,0)\in K$, $y\in \pi_{1}(W_0)$ and $0<r\leq \delta_0$. Here, $|E|$ denotes the Lebesgue measure of the set $E$ in $\mathbb{R}^{k}$ for the suitable dimension $k$.
\end{enumerate}
\end{proposition}

Thanks to Proposition \ref{prop3-3} and the estimates of type $\lambda$ operators, we obtain

\begin{lemma}
\label{lemma3-1}
For any $x_{0}\in U$, let $X=(X_{1},X_{2},\ldots,X_{m})$ be vector fields satisfying H\"{o}rmander's condition of step $s(x_0)$ at $x_{0}$, and let $W_{0}$ be an open neighborhood of $(x_0,0)\in \mathbb{R}^{n+l}$ as given by Proposition \ref{prop3-3}. Then, for any bounded open subset $W_{1}\subset \subset W_{0}$ and any function $a\in C_{0}^{\infty}(W_{1})$, there exists a positive constant $C>0$ such that for any $u\in C_{0}^{\infty}(W_1)$, we have
\begin{equation}
  |a(\xi)u(\xi)|\leq C\int_{W_{1}}(|\widetilde{X}u(\eta)|+|u(\eta)|)\frac{d\eta}{\tilde{d}(\xi,\eta)^{Q_{0}-1}}\qquad \forall \xi\in W_{1},
\end{equation}
where $\widetilde{X}=(\widetilde{X_{1}},\widetilde{X_{2}},\ldots,\widetilde{X_{m}})$ denotes the lifting vector fields of $X=(X_{1},X_{2},\ldots,X_{m})$,  $Q_{0}$ is the homogeneous dimension of $\mathbb{G}$, and
$\tilde{d}$ is the subunit metric induced by $\widetilde{X}$ in $W_0$.
\end{lemma}
\begin{proof}
Using \cite[(15.5), p. 298]{stein1976} (also see \cite[Theorem 11.19, Theorem 11.25]{Bramanti2022}), we deduce that for any  $a\in C_{0}^{\infty}(W_{1})$, there exists linear operators $T_{0},T_{1},\ldots,T_{m}$ of type $1$ such that
\begin{equation}\label{3-9}
au=\sum_{j=1}^{m}T_{j}\widetilde{X_{j}}u+T_{0}u\qquad \forall u\in C_{0}^{\infty}(W_1).
\end{equation}
Here, the operators $T_{0},T_{1},\ldots,T_{m}$ are defined in terms of some type $1$ kernels $K_{0}, K_{1},\ldots,K_{m}$ as follows:
\begin{equation}\label{3-10}
  (T_{j}f)(\xi):=\int_{W_{1}}K_{j}(\xi,\eta)f(\eta)d\eta\qquad \forall f\in C_{0}^{\infty}(W_{1}),~~0\leq j\leq m.
\end{equation}
 For precise definitions of operators and kernels of type $\lambda$, one can refer to \cite[Defintion 11.11, Defintion 11.7]{Bramanti2022} and \cite[(13.3), (13,4), pp. 288-289]{stein1976}. By means of \cite[(13.4), p. 289]{stein1976} and \cite[Theorem 11.10]{Bramanti2022}, we obtain the following growth estimate for kernels of type $1$:
\begin{equation}\label{3-11}
|K_{j}(\xi,\eta)|\leq C(\rho(\xi,\eta))^{1-Q_{0}},
\end{equation}
where $\rho(\xi,\eta)=\|\Theta(\xi,\eta)\|_{\mathbb{G}}$ is the quasidistance and $Q_{0}$ is the homogeneous dimension of $\mathbb{G}$ given by Proposition \ref{prop3-3} (see \cite[(13.3), (13.6)]{stein1976}). Then, it follows from Proposition \ref{prop3-3} (3) that $d\eta=c(\xi)(1+O(\|u\|_{\mathbb{G}}))du$, where $c(\xi)$ is a smooth function, bounded and bounded away from zero. Therefore, for any small $\varepsilon>0$ and any fixed $\xi\in W_{1}$, \eqref{3-11} derives that
\[ \int_{\rho(\xi,\eta)<\varepsilon}|K_{j}(\xi,\eta)|d\eta\leq C\int_{\|u\|_{\mathbb{G}}<\varepsilon}\|u\|_{\mathbb{G}}^{1-Q_{0}}du<\infty, \]
which  confirms the well-definedness of  $T_{0},T_{1},\ldots,T_{m}$. Since $\rho$ is equivalent to the subunit metric $\tilde{d}$ in $W_{1}$,  \eqref{3-9}-\eqref{3-11} imply that
\begin{equation*}
\begin{aligned}
|a(\xi)u(\xi)|&\leq C\int_{W_{1}}\left(\sum_{j=1}^{m}|\widetilde{X_{j}}u(\eta)|+|u(\eta)|\right)\tilde{d}(\xi,\eta)^{1-Q_{0}}d\eta\qquad \forall \xi\in W_{1}.
\end{aligned}
\end{equation*}
\end{proof}

Owing to Proposition \ref{prop3-3} and Lemma \ref{lemma3-1}, we can derive the local version of Proposition \ref{prop3-2}.
\begin{lemma}
\label{lemma3-2}
Let $X=(X_{1},X_{2},\ldots,X_{m})$ satisfy  condition (H). For any fixed point $x_0\in U$, there exists a open neighborhood $U(x_0)$ of $x_0$ such that for any $u\in C_{0}^{\infty}(U(x_0))$, we have
\begin{equation}\label{3-12}
  |u(x)|\leq C\int_{U(x_0)}\frac{d(x,y)}{|B(x,d(x,y))|}(|Xu(y)|+|u(y)|)dy, \qquad \forall x\in U(x_0).
\end{equation}

\end{lemma}
\begin{proof}
According to Proposition \ref{prop3-3}, for any fixed point $x_0\in U$, there exist lifting vector fields $\widetilde{X_{1}}, \widetilde{X_{2}},\ldots,\widetilde{X_{m}}$ defined in an open neighborhood $W_{0}$ of $(x_{0},0)\in \mathbb{R}^{n+l}$ for some $l\in \mathbb{N}$. Let $\tilde{d}$ be the subunit metric induced by $\widetilde{X_{1}}, \widetilde{X_{2}},\ldots,\widetilde{X_{m}}$ in $W_{0}$, and denote by $\tilde{B}((x_0,0),r)=\{(y,t)\in W_0|\tilde{d}((x_0,0),(y,t))<r\}$ the subunit ball in $W_0$ associated with $\tilde{d}$.
We first choose a bounded open neighborhood $W_{3}$ of $(x_0,0)$ such that  $(x_0,0)\in W_{3}\subset\subset W_{0}$ and  $\pi_{1}(\overline{W_{3}})$ is a compact subset of $U$, where $\pi_{1}(\xi)=\pi_{1}(x,s)=x$ denotes the projection from $\mathbb{R}^{n+l}$ to $\mathbb{R}^{n}$. From properties (6) and (8) in Proposition \ref{prop3-3}, there exists a positive constant $\delta_{0}$ associated with the  compact set $\overline{W_{3}}$ such that \eqref{3-6} and \eqref{3-7} hold for $K=\overline{W_{3}}$. Besides, by Proposition \ref{prop2-3}, there exist positive constants $C>1$ and $\rho_{\pi_{1}(\overline{W_{3}})}>0$ such that
\begin{equation}\label{3-13}
C^{-1}\Lambda(x,r)\leq |B(x,r)|\leq C\Lambda(x,r)\qquad \forall x\in \pi_{1}(\overline{W_{3}}),~0<r\leq \rho_{\pi_{1}(\overline{W_{3}})}.
\end{equation}
Then, we choose some positive constants $0<\delta_{2}<\delta_{1}\leq \min\{\frac{1}{2}\delta_{0},\frac{1}{3}\rho_{\pi_{1}(\overline{W_{3}})}\} $ and $T>0$ such that
\begin{equation}\label{3-14}
W_{2}:=B(x_{0},\delta_{2})\times (-T,T)^{l}\subset\subset W_{1}:=\tilde{B}((x_0,0),\delta_{1})\subset \tilde{B}((x_0,0),3\delta_{1})\subset W_{3}\subset\subset W_{0}.
\end{equation}
Let $U(x_0):=B(x_{0},\delta_{2})$ be the open neighborhood of $x_0$. We next show that \eqref{3-12} holds on $U(x_0)$.

Set the cut-off functions $a\in C_{0}^{\infty}(W_{1})$ with $a(\xi)=1$ on $\overline{W_{2}}$, and $b\in C_{0}^{{\infty}}((-T,T)^{l})$ with $b(s)=1$ on $(-\frac{1}{2}T,\frac{1}{2}T)^{l}$. For any $u\in C_{0}^{\infty}(U(x_0))$, by \eqref{3-14} we have $u(x)b(s)\in C_{0}^{{\infty}}(W_{1})$. Additionally, $W_{1}\subset \tilde{B}((x,0),2\delta_{1})$ for all $\xi=(x,0)\in W_{2}$. Note that $\widetilde{X_{j}}u=X_{j}u$. Owing to Lemma \ref{lemma3-1}, for any $x\in U(x_0)$ such that $\xi=(x,0)\in W_{2}$, we obtain
\begin{equation}
\begin{aligned}\label{3-15}
|u(x)|&=|a(x,0)u(x)b(0)|\leq C\int_{W_{1}}\left(\sum_{j=1}^{m}|\widetilde{X_{j}}(ub)(\eta)|+|(ub)(\eta)|\right)\tilde{d}(\xi,\eta)^{1-Q_{0}}d\eta\\
&=C\int_{\pi_{1}(W_{1})}\left(|Xu(y)|+|u(y)|\right)\left(\int_{\{t\in \mathbb{R}^{l}|(y,t)\in W_{1}\} }\frac{dt}{\tilde{d}((x,0),(y,t))^{Q_{0}-1}}\right)dy\\
& \leq C\int_{\pi_{1}(W_{1})}\left(|Xu(y)|+|u(y)|\right)\left(\int_{\{t\in \mathbb{R}^{l}|(y,t)\in \tilde{B}((x,0),2\delta_{1})\} }\frac{dt}{\tilde{d}((x,0),(y,t))^{Q_{0}-1}}\right)dy\\
&=C\int_{B(x_0,\delta_1)}\left(|Xu(y)|+|u(y)|\right)\left(\int_{\{t\in \mathbb{R}^{l}|(y,t)\in \tilde{B}((x,0),2\delta_{1})\}}\frac{dt}{\tilde{d}((x,0),(y,t))^{Q_{0}-1}}\right)dy,
\end{aligned}
\end{equation}
where $Q_{0}$ is the homogeneous dimension of $\mathbb{G}$.

Let us examine the integral
\begin{equation}\label{3-16}
I(x,y):=\int_{\{t\in \mathbb{R}^{l}|(y,t)\in \tilde{B}((x,0),2\delta_{1})\}}\frac{dt}{\tilde{d}((x,0),(y,t))^{Q_{0}-1}}
\end{equation}
for fixed  $(x,0)\in W_{2}$ and $y\in B(x_0,\delta_1)\subset \pi_{1}(W_{0})$. It follows from \eqref{3-14} that $ \tilde{B}((x,0),2\delta_{1})\subset \tilde{B}((x_{0},0),3\delta_{1})\subset\subset W_{0}$. Thus,  Proposition \ref{prop3-3} (5) indicates that for any $t\in \mathbb{R}^{l}$ satisfying $(y,t)\in \tilde{B}((x,0),2\delta_{1})$, we have
\begin{equation}\label{3-17}
d(x,y)\leq \tilde{d}((x,0),(y,t))<2\delta_{1}.
\end{equation}
According to \eqref{3-17}, the estimates of \eqref{3-16} can be divided into the following two cases:\\
\noindent \emph{\textbf{Case 1}:} $\delta_{1}\leq d(x,y)$. Using \eqref{3-6}, \eqref{3-7} and \eqref{3-17}, we have
\begin{equation}
\begin{aligned}\label{3-18}
I(x,y)&=\int_{\{t\in \mathbb{R}^{l}|(y,t)\in \tilde{B}((x,0),2\delta_{1})\}}\frac{dt}{\tilde{d}((x,0),(y,t))^{Q_{0}-1}}\\
&\leq \frac{1}{d(x,y)^{Q_{0}-1}}|\{t\in \mathbb{R}^{l}|(y,t)\in \tilde{B}((x,0),2\delta_{1})\}|\\
&\leq \frac{C}{d(x,y)^{Q_{0}-1}}\cdot \frac{|\tilde{B}((x,0),2\delta_{1})|}{|B(x,2\delta_{1})|}\\
&\leq \frac{C}{d(x,y)^{Q_{0}-1}}\cdot \frac{(2\delta_{1})^{Q_{0}}}{|B(x,2\delta_{1})|}\leq C\frac{d(x,y)}{|B(x,d(x,y))|},
\end{aligned}
\end{equation}
where $C>0$ is a positive constant independent of $x$ and $y$.

\noindent \emph{\textbf{Case 2}:} $\delta_{1}> d(x,y)$. In this case, we can choose an integer $k_0\geq 1$ (depends on $x,y$ and $\delta_1$), such that $2^{k_{0}}d(x,y)< 2\delta_{1}\leq 2^{k_{0}+1}d(x,y)$. Then we have
\begin{equation}
\begin{aligned}
I(x,y)&=\int_{\{t\in \mathbb{R}^{l}|(y,t)\in \tilde{B}((x,0),2\delta_{1})\}}\frac{dt}{\tilde{d}((x,0),(y,t))^{Q_{0}-1}}\\
&=\sum_{k=0}^{k_0-1}\int_{\{t\in \mathbb{R}^{l}|(y,t)\in \tilde{B}((x,0),2^{k+1}d(x,y))\setminus \tilde{B}((x,0),2^{k}d(x,y)) \}}\frac{dt}{\tilde{d}((x,0),(y,t))^{Q_{0}-1}}\\
&+\int_{\{t\in \mathbb{R}^{l}|(y,t)\in \tilde{B}((x,0),2\delta_{1})\setminus \tilde{B}((x,0),2^{k_{0}}d(x,y)) \}}\frac{dt}{\tilde{d}((x,0),(y,t))^{Q_{0}-1}}:=I_{1}(x,y)+I_{2}(x,y).
\end{aligned}
\end{equation}
By \eqref{3-6}, \eqref{3-7} and \eqref{3-13}, we obtain
\begin{equation}
\begin{aligned}\label{3-20}
I_{1}(x,y)&=\sum_{k=0}^{k_0-1}\int_{\{t\in \mathbb{R}^{l}|(y,t)\in \tilde{B}((x,0),2^{k+1}d(x,y))\setminus \tilde{B}((x,0),2^{k}d(x,y)) \}}\frac{dt}{\tilde{d}((x,0),(y,t))^{Q_{0}-1}}\\
&\leq \sum_{k=0}^{k_0-1}\frac{|\{t\in \mathbb{R}^{l}|(y,t)\in \tilde{B}((x,0),2^{k+1}d(x,y)) \}|}{(2^{k}d(x,y))^{Q_{0}-1}}\\
&\leq C\sum_{k=0}^{k_0-1} \frac{|\tilde{B}((x,0),2^{k+1}d(x,y))|}{(2^{k}d(x,y))^{Q_{0}-1}|B(x,2^{k+1}d(x,y))|}\\
&\leq C\sum_{k=0}^{k_0-1} \frac{2^{Q_{0}+k}d(x,y)}{|B(x,2^{k+1}d(x,y))|}\leq C\sum_{k=0}^{k_0-1} \frac{2^{Q_{0}+k}d(x,y)}{\sum_{I}|\lambda_{I}(x)|(2^{k+1}d(x,y))^{d(I)}}\\
&\leq  C\sum_{k=0}^{k_0-1} \frac{2^{Q_{0}+k}d(x,y)}{2^{(k+1)n}\sum_{I}|\lambda_{I}(x)|(d(x,y))^{d(I)}}\leq \frac{C2^{Q_{0}-n}d(x,y)}{(1-2^{1-n})|B(x,d(x,y))|},
\end{aligned}
\end{equation}
where $C>0$ is a positive constant independent of $x,y$ and $k_0$. Similarly, using $2^{k_{0}}d(x,y)\leq 2\delta_{1}\leq 2^{k_{0}+1}d(x,y)$, we deduce from \eqref{3-6}, \eqref{3-7} and \eqref{3-13} that
\begin{equation}
\begin{aligned}\label{3-21}
I_{2}(x,y)&=\int_{\{t\in \mathbb{R}^{l}|(y,t)\in \tilde{B}((x,0),2\delta_{1})\setminus \tilde{B}((x,0),2^{k_{0}}d(x,y)) \}}\frac{dt}{\tilde{d}((x,0),(y,t))^{Q_{0}-1}}\\
&\leq \frac{|\{t\in \mathbb{R}^{l}|(y,t)\in \tilde{B}((x,0),2\delta_{1}) \}|}{(2^{k_0}d(x,y))^{Q_{0}-1}}\leq \frac{1}{(2^{k_0}d(x,y))^{Q_{0}-1}}\frac{|\tilde{B}((x,0),2\delta_{1})|}{|B(x,2\delta_{1})|}\\
&\leq \frac{C}{(2^{k_0}d(x,y))^{Q_{0}-1}}\frac{(2\delta_{1})^{Q_{0}}}{|B(x,2^{k_{0}}d(x,y))|}\leq \frac{C}{(2^{k_0}d(x,y))^{Q_{0}-1}}\frac{(2^{k_{0}+1}d(x,y))^{Q_{0}}}{\sum_{I}|\lambda_{I}(x)|(2^{k_{0}}d(x,y))^{d(I)}}\\
&\leq \frac{C}{(2^{k_0}d(x,y))^{Q_{0}-1}}\frac{(2^{k_{0}+1}d(x,y))^{Q_{0}}}{2^{k_{0}n}\sum_{I}|\lambda_{I}(x)|d(x,y)^{d(I)}}\leq\frac{C2^{Q_0}d(x,y)}{|B(x,d(x,y))|},
\end{aligned}
\end{equation}
where $C>0$ is a positive constant independent of $x,y$ and $k_0$.

Hence, we conclude from \eqref{3-15}-\eqref{3-21} that
\[   |u(x)|\leq C\int_{U(x_0)}\frac{d(x,y)}{|B(x,d(x,y))|}(|Xu(y)|+|u(y)|)dy \qquad \forall x\in U(x_0), \]
where $U(x_0)=B(x_0,\delta_{2})\subset B(x_0,\delta_1)$, and $C>0$ is a positive constant independent of $x$.
\end{proof}

Now, we can complete the proof of Proposition \ref{prop3-2}.

\begin{proof}[Proof of Proposition \ref{prop3-2}]
Since $\overline{W}$ is a compact subset of $U$, Lemma \ref{lemma3-2} allows us to choose a finite open cover $\{U(x_{i})\}_{i=1}^{l_{0}}$ of $\overline{W}$ and corresponding positive constants $\{C(x_{i})\}_{i=1}^{l_{0}}$ such that for any $u\in C_{0}^{\infty}(U(x_{i}))$,
\begin{equation}\label{3-22}
  |u(x)|\leq C(x_i)\int_{U(x_i)}\frac{d(x,y)}{|B(x,d(x,y))|}(|Xu(y)|+|u(y)|)dy, \qquad \forall x\in U(x_i).
\end{equation}
By partition of unity, there exists a collection of cut-off functions $\{\psi_{i}\}_{i=1}^{l_{0}}$ such that $\psi_{i}\in C_{0}^{\infty}(U(x_{i}))$, $0\leq\psi_{i}\leq 1$, and $\sum_{i=1}^{l_{0}}\psi_{i}(x)=1$ for all $x\in \overline{W}$. As a result of \eqref{3-22}, for any $f\in C_{0}^{\infty}(W)$,
\begin{equation*}
\begin{aligned}
|f(x)|&\leq \sum_{i=1}^{l_{0}}|\psi_{i}(x)f(x)|\leq \sum_{i=1}^{l_{0}}C(x_i)\int_{U(x_i)}\frac{d(x,y)}{|B(x,d(x,y))|}(|X(\psi_{i}f)|+|\psi_{i}f|)dy\\
&\leq C\int_{W}\frac{d(x,y)}{|B(x,d(x,y))|}(|Xf(y)|+|f(y)|)dy\qquad \forall x\in \overline{W},
\end{aligned}
\end{equation*}
where $C>0$ is a positive independent of $x$.
\end{proof}
\begin{remark}
In general, the $f$ term on the right-hand side of \eqref{3-2} cannot be eliminated due to  the locality of Lemma \ref{lemma3-2} and the partition of unity arguments.
\end{remark}

\subsection{Weighted weak-$L^p$ estimates of $T$ operators}

Let $W\subset\subset U$ be a bounded open subset, and $\Lambda(x,r)$ denotes the Nagel-Stein-Wainger polynomial mentioned in \eqref{2-3} above. Next, we are devoted to the weighted weak-$L^p$ estimates of the following linear operators:
\begin{equation}\label{T}
  Tf(x):=\int_{W}\frac{d(x,y)}{\Lambda(x,d(x,y))}f(y)dy.
\end{equation}

For this purpose, we employ an abstract lemma from harmonic analysis. Consider a measure space $(\mathcal{X},\mathcal{A},\mu)$, where $\mathcal{X}$ is a set, $\mathcal{A}$ is a $\sigma$-algebra on $\mathcal{X}$, and $\mu$ is a positive measure defined on $(\mathcal{X},\mathcal{A})$. Then, we have

\begin{lemma}
\label{lemma3-3}
Let $K(x,y)$ be a measurable function on $\mathcal{X}\times \mathcal{X}$ such that for some $r>1$,
$K(\cdot,y)$ is weak-$L^{r}$ uniformly in $y$ and
\begin{equation*}
 \sup_{y\in \mathcal{X}}\|K(\cdot,y)\|_{L^{r}_{w}(\mathcal{X},d\mu)}<\infty.
\end{equation*}
Then the operator
\begin{equation*}
  Tf=\int_{\mathcal{X}}K(\cdot,y)f(y)d\mu(y)
\end{equation*}
is bounded from $L^{1}(\mathcal{X},d\mu)$ to  $L^{r}_{w}(\mathcal{X},d\mu)$. Moreover,
\begin{equation*}
\|Tf\|_{L^{r}_{w}(\mathcal{X},d\mu)}\leq \frac{2^{2-\frac{1}{r}}}{(r-1)^{\frac{1}{r}}}\|f\|_{L^1(\mathcal{X},d\mu)}\sup_{y\in \mathcal{X}}\|K(\cdot,y)\|_{L^{r}_{w}(\mathcal{X},d\mu)}.
\end{equation*}
Here, $L^{r}_{w}(\mathcal{X},d\mu)$ is the weak-$L^{r}$ space on $\mathcal{X}$ with respect to the measure $\mu$, and $\|f\|_{L^{r}_{w}(\mathcal{X},d\mu)}$ denotes the weak-$L^{r}$ norm in $L^{r}_{w}(\mathcal{X},d\mu)$.
\end{lemma}
Lemma \ref{lemma3-3} is a refinement of \cite[Lemma 15.3]{Folland-Stein1974}. We present the proof as follows:
\begin{proof}[Proof of Lemma \ref{lemma3-3}]
Without loss of generality, we can suppose that $\|f\|_{L^1(\mathcal{X},d\mu)}=1$. For any $s>0$, we let $K(x,y)=K_{1}(x,y)+K_{2}(x,y)$, where
\[ K_{1}(x,y)=\left\{
                \begin{array}{ll}
                 K(x,y) , & \hbox{$| K(x,y)|\geq\frac{s}{2}$;} \\
                  0, & \hbox{$| K(x,y)|<\frac{s}{2}$,}
                \end{array}
              \right.\quad \mbox{and}\quad K_{2}(x,y)=\left\{
                \begin{array}{ll}
               0 , & \hbox{$| K(x,y)|\geq\frac{s}{2}$;} \\
                   K(x,y), & \hbox{$| K(x,y)|<\frac{s}{2}$.}
                \end{array}
              \right.\]
We then define $T_{1}f:=\int_{\mathcal{X}}K_{1}(\cdot,y)f(y)d\mu(y)$, $T_{2}f:=\int_{\mathcal{X}}K_{2}(\cdot,y)f(y)d\mu(y)$, and adopt the notation
\[
\beta_{Tf}(s):=\mu(\{x\in \mathcal{X}||Tf(x)|>s\}) \quad \forall s>0.\]
It follows that
$\beta_{Tf}(2s)\leq \beta_{T_{1}f}(s)+\beta_{T_{2}f}(s)$.
Since
\[ |T_{2}f(x)|\leq \int_{\mathcal{X}}|K_{2}(x,y)||f(y)|d\mu(y) \leq \frac{s}{2},\]
we have $\{x\in \mathcal{X}| |T_{2}f(x)|>s\}=\varnothing$ and $\beta_{T_{2}f}(s)=0$ for any $s>0$.

Next, for any fixed $y\in \mathcal{X}$, we have
\[ \begin{aligned}
\int_{\mathcal{X}}|K_{1}(x,y)|d\mu(x)&=\int_{\frac{s}{2}}^{+\infty}\beta_{K(\cdot,y)}(t)dt\leq \int_{\frac{s}{2}}^{+\infty}\frac{\|K(\cdot,y)\|_{L^{r}_{w}(\mathcal{X},d\mu)}^{r}}{t^{r}}dt\\
&=\frac{1}{r-1}\left(\frac{s}{2}\right)^{1-r}\|K(\cdot,y)\|_{L^{r}_{w}(\mathcal{X},d\mu)}^{r}.
\end{aligned}\]
Hence,
\[ \begin{aligned}
\|T_{1}f\|_{L^{1}(\mathcal{X},d\mu)}&=\int_{\mathcal{X}}|T_{1}f(x)|d\mu(x)\leq \int_{\mathcal{X}}|f(y)|\left(\int_{\mathcal{X}}|K_{1}(x,y)|d\mu(x)\right)d\mu(y)\\
&\leq \left[\sup_{y\in \mathcal{X}}\int_{\mathcal{X}}|K_{1}(x,y)|d\mu(x)\right]\cdot\int_{\mathcal{X}}|f(y)|d\mu(y)\\
&\leq \frac{1}{r-1}\left(\frac{s}{2}\right)^{1-r}\sup_{y\in \mathcal{X}}\|K(\cdot,y)\|_{L^{r}_{w}(\mathcal{X},d\mu)}^{r}.
\end{aligned} \]
As a result, we have for any $s>0$,
\[ \beta_{Tf}(2s)\leq \beta_{T_{1}f}(s)\leq \frac{\|T_{1}f\|_{L^{1}(\mathcal{X},d\mu)}}{s}\leq
\frac{1}{(r-1)2^{1-r}}s^{-r}\left(\sup_{y\in \mathcal{X}}\|K(\cdot,y)\|_{L^{r}_{w}(\mathcal{X},d\mu)}\right)^{r}, \]
which yields that
\[ \|Tf\|_{L^{r}_{w}(\mathcal{X},d\mu)}=\sup_{s>0}\bigg((2s)[\beta_{Tf}(2s)]^{\frac{1}{r}}\bigg) \leq \frac{2^{2-\frac{1}{r}}}{(r-1)^{\frac{1}{r}}}\sup_{y\in \mathcal{X}}\|K(\cdot,y)\|_{L^{r}_{w}(\mathcal{X},d\mu)}.\]
\end{proof}

According to Lemma \ref{lemma3-3}, we have
\begin{proposition}
\label{prop3-4}
Let $X=(X_{1},X_{2},\ldots,X_{m})$  satisfy condition (H). Suppose that $W\subset \subset U$ is a bounded open subset. Then, for every $n$-tuple $I=(i_{1},i_{2},\ldots,i_{n})$ with $1\leq i_{j}\leq l$,
the linear operator
\begin{equation*}
  Tf=\int_{W}\frac{d(\cdot,y)}{\Lambda(\cdot,d(x,y))}f(y)dy
\end{equation*}
is bounded from $L^1(W)$ to $L^{\frac{d(I)}{d(I)-1}}_w(W,|\lambda_I|^{\frac{1}{d(I)-1}}dx)$, that is,
  \begin{equation}\label{3-24}
    \int_{\{x\in W||Tf(x)|>t\}}|\lambda_I(x)|^{\frac{1}{d(I)-1}}dx\leq C t^{-\frac{d(I)}{d(I)-1}}\|f\|_{L^1(W)}^{\frac{d(I)}{d(I)-1}} \qquad \forall f\in L^{1}(W),~t>0,
  \end{equation}
where  $C>0$ is a positive constant.
\end{proposition}
\begin{proof}
 Let $r:=\frac{d(I)}{d(I)-1}>1$ and $K(x,y):=\frac{d(x,y)}{\Lambda(x,d(x,y))}$ be a measurable function on $W\times W$. We first show that there exists a positive constant $C>0$ such that for all $y\in W$ and $t>0$,
\begin{equation}\label{3-25}
  \int_{A_t(y)}|\lambda_I(x)|^{\frac{r}{d(I)}}dx\leq Ct^{-r},
\end{equation}
where  $A_t(y):=\left\{x\in W\Big| K(x,y) >t\right\}$. Clearly, \eqref{3-25} is equivalent to
\begin{equation}\label{3-26}
  \sup_{y\in W}\|K(\cdot,y)\|_{L_{w}^{r}(W,d\mu_{I})}<\infty
\end{equation}
with $d\mu_I=|\lambda_I(x)|^{\frac{1}{d(I)-1}}dx$.

Applying  Proposition \ref{prop2-3} and Proposition \ref{prop2-4} for the compact subset $\overline{W}$, there exist $\rho_{\overline{W}}>0$ and $C^{*}\geq 1$ such that for any $x,y\in \overline{W}$ with $d(x,y)\leq \frac{\rho_{\overline{W}}}{2}$, we have
\begin{equation}\label{3-27}
 \frac{1}{C^{*}}\Lambda(y,d(x,y))\leq \Lambda(x,d(x,y))\leq C^{*}\Lambda(y,d(x,y)).
\end{equation}
If $\delta_{\overline{W}}:=\sup_{x,y\in \overline{W}}d(x,y)\leq \frac{\rho_{\overline{W}}}{2}$, the restriction $d(x,y)\leq \frac{\rho_{\overline{W}}}{2}$ can be removed, and \eqref{3-27} holds for all $x,y\in \overline{W}$. Now, let us consider the case where $\delta_{\overline{W}}>\frac{\rho_{\overline{W}}}{2}$ and $\frac{\rho_{\overline{W}}}{2}\leq d(x,y)\leq \delta_{\overline{W}}$. Due to Proposition \ref{prop2-2}, we have
\[ \Lambda(x,d(x,y))\leq \max_{x\in \overline{W}}\Lambda(x,\delta_{\overline{W}})\leq C\min_{x\in \overline{W}}\Lambda\left(y,\frac{\rho_{\overline{W}}}{2}\right)\leq C\Lambda(x,d(x,y)).\]
This implies \eqref{3-27} holds for all $x,y\in \overline{W}$.

It follows \eqref{3-27} from that for any $y\in W$ and $t>0$,
\begin{equation}\label{3-28}
  A_t(y)=\left\{x\in W\bigg| K(x,y) >t\right\}\subset\left\{x\in W\bigg|\frac{d(x,y)}{\Lambda(y,d(x,y))}> \frac{t}{C^{*}}\right\}:=B_{t}(y).
\end{equation}
Then, we have\\
\noindent \emph{\textbf{Case 1}:} $0<t\leq\frac{C^{*}\frac{\rho_{\overline{W}}}{2}}{\Lambda(y,\frac{\rho_{\overline{W}}}{2})}$. By Proposition \ref{prop2-3}, we obtain for any $y\in W$,
\begin{equation}
\begin{aligned}
  t^{r}\int_{A_{t}(y)}|\lambda_{I}(x)|^{\frac{r}{d(I)}}dx&\leq \left(\frac{C^{*}\frac{\rho_{\overline{W}}}{2}}{\Lambda(y,\frac{\rho_{\overline{W}}}{2})}\right)^{r}|W|\cdot \max_{x\in\overline{W}}|\lambda_{I}(x)|^{\frac{r}{d(I)}}\\
  &\leq  \frac{C(C^{*})^{\frac{d(I)}{d(I)-1}}|W|}{\rho_{\overline{W}}^{r(\tilde{\nu}(\overline{W})-1)}}\cdot \max_{x\in\overline{W}}|\lambda_{I}(x)|^{\frac{1}{d(I)-1}},
\end{aligned}
\end{equation}
which yields \eqref{3-25}.

\noindent \emph{\textbf{Case 2}:} $t>\frac{C^{*}\frac{\rho_{\overline{W}}}{2}}{\Lambda(y,\frac{\rho_{\overline{W}}}{2})}$. For any $y\in W$, since $g(s):= \frac{s}{\Lambda(y,s)}$ is a strict decreasing function on $\mathbb{R}^{+}$ with $\lim_{s\to+\infty}g(s)=0$, there exists a unique $0<s_{t}<\frac{\rho_{\overline{W}}}{2}$ such that
\begin{equation}\label{3-30}
  \frac{s_{t}}{\Lambda(y,s_{t})}=\frac{t}{C^{*}}>\frac{\frac{\rho_{\overline{W}}}{2}}{\Lambda(y,\frac{\rho_{\overline{W}}}{2})},
\end{equation}
and
\begin{equation}\label{3-31}
B_{t}(y)=\left\{x\in W\bigg|\frac{d(x,y)}{\Lambda(y,d(x,y))}> \frac{t}{C^{*}}\right\}\subset \{x\in W|d(x,y)<s_{t}\}\subset B(y,s_{t}).
\end{equation}
According to \eqref{2-11} and \eqref{3-30},
\begin{equation}\label{3-32}
\frac{|B(y,s_{t})|}{s_{t}}\leq \frac{C_{2}\Lambda(y,s_{t})}{s_{t}}=\frac{C_{2}C^{*}}{t},
\end{equation}
where $C_{2}>0$ is the positive constant in \eqref{2-11}. Furthermore, for any $x\in B(y,s_{t})$, we have $B(x,s_{t})\subset B(y,2s_{t})$. Therefore, Proposition \ref{prop2-3}, Proposition \ref{prop2-4}, \eqref{3-28} and \eqref{3-31} imply that
\begin{equation}\label{3-33}
\begin{aligned}
|\lambda_{I}(x)|&\leq s_{t}^{-d(I)}\Lambda(x,s_{t})\leq C_{1}^{-1}s_{t}^{-d(I)}|B(x,s_{t})|\leq C_{1}^{-1}s_{t}^{-d(I)}|B(y,2s_{t})|\\
&\leq C_{1}^{-1}C_{3}s_{t}^{-d(I)}|B(y,s_{t})|
\end{aligned}
\end{equation}
holds for any $y\in W$, $x\in A_{t}(y)$ and $t>\frac{C^{*}\frac{\rho_{\overline{W}}}{2}}{\Lambda(y,\frac{\rho_{\overline{W}}}{2})}$, where $C_{1} $and  $C_{3}$ are the positive constants in \eqref{2-10} and \eqref{2-14}, respectively. Using \eqref{3-28}, \eqref{3-32}, \eqref{3-33} and Proposition \ref{prop2-4}, we conclude that for any $y\in W$,
\begin{equation*}
\begin{aligned}
\int_{A_t(y)}|\lambda_I(x)|^{\frac{r}{d(I)}}dx&\leq (C_{1}^{-1}C_{3})^{\frac{1}{d(I)-1}}\left( \frac{|B(y,s_{t})|}{s_{t}}\right)^{r}\\
&\leq \frac{(C_{1}^{-1}C_{3})^{\frac{1}{d(I)-1}}(C_{2}C^{*})^{\frac{d(I)}{d(I)-1}}}{t^{r}}\qquad \forall t>\frac{C^{*}\frac{\rho_{\overline{W}}}{2}}{\Lambda(y,\frac{\rho_{\overline{W}}}{2})},
\end{aligned}
\end{equation*}
which also yields \eqref{3-25}.

Now, employing Lemma \ref{lemma3-3} for $d\mu_I=|\lambda_I(x)|^{\frac{1}{d(I)-1}}dx$ and $\mathcal{X}=W$, it follows that
\[ \|Tf\|_{L^{r}_{w}(W,d\mu_I)}\leq C\|f\|_{L^1(W,d\mu_I)}\leq C \max_{x\in\overline{W}}|\lambda_{I}(x)|^{\frac{r}{d(I)}}\int_{W}|f(x)|dx \]
for all $f\in L^{1}(W)$, which proves \eqref{3-24}.
\end{proof}

\section{Proofs of main results}
\label{section4}

\subsection{Proof of Theorem \ref{thm1}}

By means of  Proposition \ref{prop3-2}, Proposition \ref{prop3-4} and Lemma \ref{lemma3-3}, we have

\begin{lemma}
\label{lemma4-1}
Let $X=(X_{1},X_{2},\ldots,X_{m})$  satisfy assumption (H). Suppose that $\Omega\subset \subset U$ is a bounded open subset. Then, there exists a positive constant $C>0$ such that for every $n$-tuple $I=(i_{1},i_{2},\ldots,i_{n})$ with $1\leq i_{j}\leq l$,
\begin{equation}\label{4-1}
\int_{\{x\in \Omega||u(x)|>t\}}|\lambda_I(x)|^{\frac{1}{d(I)-1}}dx\leq Ct^{-\frac{d(I)}{d(I)-1}}\|Xu\|_{L^1(\Omega)}^{\frac{d(I)}{d(I)-1}}\qquad \forall u\in \mathcal{W}^{1,1}_{X,0}(\Omega),~~ t>0.
\end{equation}
\end{lemma}
\begin{proof}
For any $u\in C_{0}^\infty(\Omega)$, we have $(|Xu|+|u|)\in L^{1}(\Omega)$. Let
\begin{equation*}
  Tf(x):=\int_{\Omega}\frac{d(x,y)}{\Lambda(x,d(x,y))}f(y)dy.
\end{equation*}
By Proposition \ref{prop3-2} and \eqref{2-10} we have
\begin{equation}\label{4-2}
  |u(x)|\leq C\int_{\Omega}\frac{d(x,y)}{|B(x,d(x,y))|}(|Xu(y)|+|u(y)|)dy\leq CT(|Xu|+|u|)(x)\qquad \forall x\in \Omega.
\end{equation}
This means
\begin{equation}\label{4-3}
  \{x\in \Omega||u(x)|>t\}\subset  \{x\in \Omega||T(|Xu|+|u|)(x)|>C^{-1}t\}\qquad \forall t>0.
\end{equation}
Since the number of $n$-tuples $I=(i_{1},i_{2},\ldots,i_{n})$ with $1\leq i_{j}\leq l$ is finite, by Proposition \ref{prop3-4} and \eqref{4-2} we can find a positive constant $C>0$ such that for all $n$-tuple $I$,
\begin{equation}
\begin{aligned}\label{4-4}
\int_{\{x\in \Omega||u(x)|>t\}}|\lambda_I(x)|^{\frac{1}{d(I)-1}}dx&\leq \int_{\{x\in \Omega||T(|Xu|+|u|)(x)|>C^{-1}t\}}|\lambda_I(x)|^{\frac{1}{d(I)-1}}dx\\
&\leq Ct^{-\frac{d(I)}{d(I)-1}}\|Xu\|_{L^1(\Omega)}^{\frac{d(I)}{d(I)-1}}\qquad \forall u\in C_{0}^{\infty}(\Omega).
 \end{aligned}
 \end{equation}

It remains to extend \eqref{4-4} to all $u\in \mathcal{W}^{1,1}_{X,0}(\Omega)$.  Clearly, for each $n$-tuple $I$ and any measurable subset $E$ of $\Omega$, we have
\[ \mu_{I}(E):=\int_{E}|\lambda_I(x)|^{\frac{1}{d(I)-1}} dx\leq \max_{x\in \overline{\Omega}}|\lambda_I(x)|^{\frac{1}{d(I)-1}}|E|. \]
 For any $u\in \mathcal{W}^{1,1}_{X,0}(\Omega)$, there exists a sequence $\{u_k\}_{k=1}^{\infty}\subset C_0^\infty(\Omega)$ such that $u_k\to u$ in $\mathcal{W}^{1,1}_{X,0}(\Omega)$. Then, we can select a subsequence $\{u_{k_j}\}_{j=1}^{\infty}\subset \{u_k\}_{k=1}^{\infty}$ such that for any $\varepsilon>0$,
 \begin{equation}\label{4-5}
\lim_{j\to\infty}|\{x\in \Omega||u_{k_j}(x)-u(x)|\geq \varepsilon \}|=0.
 \end{equation}
  Now, for any $t>0$, by \eqref{4-5} we have
 \begin{equation}\label{4-6}
  \begin{aligned}
  &\int_{\{x\in \Omega||u(x)|>t\}}|\lambda_I(x)|^{\frac{1}{d(I)-1}}dx=\mu_I\left(\left\{x\in \Omega  ||u(x)|>t\right\}\right)\\
  &\leq \mu_I\left(\left\{x\in \Omega\Big||u(x)-u_{k_j}(x)|>\frac{t}{2}\right\}\right)+\mu_I\left(\left\{x\in \Omega\Big||u_{k_j}(x)|>\frac{t}{2}\right\}\right)\\
  &\leq \max_{x\in \overline{\Omega}}|\lambda_I(x)|^{\frac{1}{d(I)-1}}\left|\left\{x\in \Omega\Big||u(x)-u_{k_j}(x)|>\frac{t}{2}\right\}\right|+
  Ct^{-\frac{d(I)}{d(I)-1}}\|Xu_{k_j}\|_{L^1(\Omega)}^{\frac{d(I)}{d(I)-1}}.
  \end{aligned}
  \end{equation}
  Letting $j\to \infty$ in \eqref{4-6} and using \eqref{4-4}, the conclusion follows.
\end{proof}

Thanks to Lemma \ref{lemma4-1}, we can obtain that

\begin{lemma}
\label{lemma4-2}
Let $X=(X_{1},X_{2},\ldots,X_{m})$  satisfy assumption (H). Suppose that $\Omega\subset \subset U$ is a bounded open subset. Then, there exists a positive constant $C>0$ such that for any $n$-tuple $I$,
\begin{equation}\label{4-7}
\int_{\Omega} |u(x)|^{\frac{d(I)}{d(I)-1}} |\lambda_I(x)|^{\frac{1}{d(I)-1}}dx\leq C\|Xu\|_{L^1(\Omega)}^{\frac{d(I)}{d(I)-1}}\qquad \forall u\in C_{0}^{\infty}(\Omega).
\end{equation}

\end{lemma}
\begin{proof}
Let $\phi(t)=\max\{0,\min\{t,1\}\}=t_+-(t-1)_+\geq 0$ be the auxiliary function on $\mathbb{R}$. For any $u\in C^\infty_0(\Omega)$ and $i\in \mathbb{Z}$, we define
  \begin{equation}\label{4-8}
u_i(x):=\phi(2^{1-i}|u(x)|-1)=\left\{
                            \begin{array}{ll}
                              0, & \hbox{if $|u(x)|\leq 2^{i-1}$;} \\[1.5mm]
                              2^{1-i}|u(x)|-1, & \hbox{if $2^{i-1}<|u(x)|\leq 2^{i}$;} \\[1.5mm]
                              1, & \hbox{if $|u(x)|>2^{i}$.}
                            \end{array}
                          \right.
\end{equation}
From Proposition \ref{prop2-9}, we see that $u_i\in \mathcal{W}^{1,1}_{X,0}(\Omega)$.  Besides,
   \begin{equation*}
Xu_i= \begin{cases}2^{1-i}{\rm sgn}(u)Xu & \text { if }2^{i-1} < |u(x)| \leq 2^i,\\ 0 & \text { otherwise. }\end{cases}
\end{equation*}
Applying Lemma \ref{lemma4-1} to $u_i$ and using \eqref{2-26} we have
\begin{equation}
\begin{aligned}\label{4-9}
\int_{\{x\in\Omega|u_i(x)>t\}}|\lambda_I(x)|^{\frac{1}{d(I)-1}}dx &\leq Ct^{-\frac{d(I)}{d(I)-1}}\|Xu_i\|_{L^1(\Omega)}^{\frac{d(I)}{d(I)-1}}\\
&\leq Ct^{-\frac{d(I)}{d(I)-1}}2^{(1-i)\frac{d(I)}{d(I)-1}}\left(\int_{\{x\in\Omega|2^{i-1} < |u(x)| \leq 2^i\}}|Xu|dx\right)^{\frac{d(I)}{d(I)-1}}.
\end{aligned}
 \end{equation}
According to \eqref{4-8} and \eqref{4-9}, we deduce that
  \begin{align*}
    &\int_{\Omega}|u(x)|^{\frac{d(I)}{d(I)-1}}|\lambda_I(x)|^{\frac{1}{d(I)-1}}dx=\sum_{i=-\infty}^{\infty}\int_{\{x\in\Omega|2^{i}<|u(x)|\leq 2^{i+1}\}}|u(x)|^{\frac{d(I)}{d(I)-1}} |\lambda_I(x)|^{\frac{1}{d(I)-1}}dx\\
&\leq\sum_{i=-\infty}^\infty  2^{(i+1)\frac{d(I)}{d(I)-1}}\int_{\{x\in\Omega|2^i<|u(x)|\leq 2^{i+1}\}}|\lambda_I(x)|^{\frac{1}{d(I)-1}}dx\\
    &\leq\sum_{i=-\infty}^\infty  2^{(i+1)\frac{d(I)}{d(I)-1}}\int_{\{x\in\Omega|u_i(x)= 1\}}|\lambda_I(x)|^{\frac{1}{d(I)-1}}dx\\
     &\leq \sum_{i=-\infty}^\infty  2^{(i+1)\frac{d(I)}{d(I)-1}}\int_{\{x\in\Omega|u_i(x)> \frac{1}{2}\}}|\lambda_I(x)|^{\frac{1}{d(I)-1}}dx\\
     &\leq C\sum_{i=-\infty}^\infty  2^{(i+1)\frac{d(I)}{d(I)-1}}2^{\frac{d(I)}{d(I)-1}}2^{(1-i)\frac{d(I)}{d(I)-1}}\left(\int_{\{x\in\Omega|2^{i-1} < |u(x)| < 2^i\}}|Xu|dx\right)^{\frac{d(I)}{d(I)-1}}\leq C\|Xu\|_{L^1(\Omega)}^{\frac{d(I)}{d(I)-1}}
  \end{align*}
for any $u\in C_{0}^{\infty}(\Omega)$, where $C>0$ is a positive constant independent of $I$.
\end{proof}

We now present the proof of Theorem \ref{thm1}.
\begin{proof}[Proof of Theorem \ref{thm1}]
We first establish \eqref{1-8} for $k=1$, which asserts for $1\leq p<\tilde{\nu}$,
\begin{equation}\label{4-10}
\|u\|_{L^{\frac{\tilde{\nu} p}{\tilde{\nu}-p}}(\Omega)}\leq C\|Xu\|_{L^p(\Omega)}\qquad\forall u\in \mathcal{W}^{1,p}_{X,0}(\Omega).
\end{equation}

Since $\overline{\Omega}$ is a compact subset of $U$, by Proposition \ref{prop2-1} and condition (H) we can choose a finite collection $\{(x_i,I_i,\lambda_{I_i},U_i)|i=1,\ldots,l_{1}\}$ such that
\begin{itemize}
  \item $x_{1},\ldots,x_{l_{1}}$ are some points in $\overline{\Omega}$;
  \item $U_{i}$ is a bounded open neighborhood of $x_i$, and $\overline{\Omega}\subset \bigcup_{i=1}^{l_{1}}U_{i}$;
  \item For each $1\leq i\leq l_{1}$, $I_{i}$ is a $n$-tuple satisfying $d(I_{i})=\nu(x_{i})\leq \tilde{\nu}$;
\item $C^{-1}\leq |\lambda_{I_i}(y)|\leq C$ holds for all $1\leq i\leq l_{1}$ and $y\in U_{i}$, where $C>1$ is a positive constant independent of $I_{i}$.
\end{itemize}
Therefore, using H\"{o}lder's inequality and  Lemma \ref{lemma4-2}, we have for any $u\in C_0^\infty(\Omega)$,
\begin{equation}\label{4-11}
\begin{aligned}
  \left(\int_\Omega|u(x)|^{\frac{\tilde{\nu}}{\tilde{\nu}-1}}dx\right)^{\frac{\tilde{\nu}-1}{\tilde{\nu}}}&
  \leq C\sum_{i=1}^{l_{1}}\left(\int_{U_i\cap\Omega}|u(x)|^{\frac{\tilde{\nu}}{\tilde{\nu}-1}}dx\right)^{\frac{\tilde{\nu}-1}{\tilde{\nu}}} \leq C\sum_{i=1}^{l_{1}}\left(\int_{U_i\cap\Omega}|u(x)|^{\frac{d(I_i)}{d(I_i)-1}}dx\right)^{\frac{d(I_i)-1}{d(I_i)}}\\
   &\leq C\sum_{i=1}^{l_{1}}\left(\int_{U_i\cap\Omega}|u(x)|^{\frac{d(I_i)}{d(I_i)-1}}|\lambda_{I_i}(x)|^{\frac{1}{d(I_i)-1}}dx\right)^{\frac{d(I_i)-1}{d(I_i)}}\\
 &\leq C \sum_{i=1}^{l_{1}}\left(\int_{\Omega}|u(x)|^{\frac{d(I_i)}{d(I_i)-1}}|\lambda_{I_i}(x)|^{\frac{1}{d(I_i)-1}}dx\right)^{\frac{d(I_i)-1}{d(I_i)}}\leq C \|Xu\|_{L^1(\Omega)}.
\end{aligned}
\end{equation}
Observing that for any $u\in \mathcal{W}_{X,0}^{1,1}(\Omega)$, there exists a sequence $\{u_{k}\}_{k=1}^{\infty}\subset C_{0}^{\infty}(\Omega)$ such that $u_{k}\to u$ in $\mathcal{W}_{X,0}^{1,1}(\Omega)$. From \eqref{4-11} we have
\begin{equation}\label{4-12}
  \|u_{k}\|_{L^{\frac{\tilde{\nu}}{\tilde{\nu}-1}}(\Omega)}\leq C\|Xu_{k}\|_{L^1(\Omega)}\qquad \forall k\geq 1,
\end{equation}
and $\{u_{k}\}_{k=1}^{\infty}$ forms a Cauchy sequence in $L^{\frac{\tilde{\nu}}{\tilde{\nu}-1}}(\Omega)$. Assume that $u_{k}\to \tilde{u}\in L^{\frac{\tilde{\nu}}{\tilde{\nu}-1}}(\Omega)$. Note that $L^{\frac{\tilde{\nu}}{\tilde{\nu}-1}}(\Omega)\subset L^{1}(\Omega)$ and $\mathcal{W}_{X,0}^{1,1}(\Omega)\subset L^{1}(\Omega)$, we obtain that $\tilde{u}=u$ a.e. on $\Omega$. Thus, taking $k\to \infty$ in \eqref{4-12}, we get
\begin{equation}\label{4-13}
 \|u\|_{L^{\frac{\tilde{\nu}}{\tilde{\nu}-1}}(\Omega)}\leq C\|Xu\|_{L^1(\Omega)}\qquad \forall u\in \mathcal{W}_{X,0}^{1,1}(\Omega),
\end{equation}
which gives the case $p=1$ of \eqref{4-10}.

We next address \eqref{4-10} for the case $1<p<\tilde{\nu}$. For any $u\in C_0^\infty(\Omega)$ and $r>1$, $|u|^r\in C_0^1(\Omega)\subset \mathcal{W}^{1,1}_{X,0}(\Omega)$. Setting $r=\frac{p(\tilde{\nu}-1)}{\tilde{\nu}-p}>1$, \eqref{4-13} and H\"{o}lder's inequality yield that for any $u\in C_0^\infty(\Omega)$,
\[
\begin{aligned}
\|u\|_{L^{\frac{\tilde{\nu}p}{\tilde{\nu}-p}}(\Omega)}^{r}&=\||u|^r\|_{L^{\frac{\tilde{\nu}}{\tilde{\nu}-1}}(\Omega)}
\leq C\int_{\Omega}|X|u|^r|dx\leq C\int_{\Omega}|u|^{r-1}|Xu|dx \leq C\|u\|_{L^{\frac{\tilde{\nu}p}{\tilde{\nu}-p}}(\Omega)}^{\frac{\tilde{\nu}(p-1)}{\tilde{\nu}-p}}\|Xu\|_{L^p(\Omega)},
\end{aligned}
\]
which implies that $\|u\|_{L^{\frac{\tilde{\nu}p}{\tilde{\nu}-p}}(\Omega)}\leq C\|Xu\|_{L^p(\Omega)}$
 holds for any $u\in C_{0}^\infty(\Omega)$ and $1<p<\tilde{\nu}$. Using the similar approximation arguments as above, we obtain \eqref{4-10} for $1<p<\tilde{\nu}$.

If $k\geq 2$, $kp< \tilde{\nu}$ and  $\frac{1}{q}=\frac{1}{p}-\frac{k}{\tilde{\nu}}$, we denote by $q_j=\frac{\tilde{\nu}p}{\tilde{\nu}-jp}$ for $j=1,\ldots,k$. Note that $1\leq p<q_{j}<\tilde{\nu}$ for $j=1,\ldots,k-1$ and $q_{k}=q$.  According to \eqref{4-10}, we have
\[
\begin{aligned}
\sum_{|J|=k}\|X^J u\|_{L^p(\Omega)}&=\sum_{j=1}^{m}\sum_{|J|=k-1}\|X_{j}(X^{J} u) \|_{L^p(\Omega)}\geq C\sum_{|J|=k-1}\|  X^{J} u\|_{L^{q_1}(\Omega)}\\
&\geq C\sum_{|J|=k-2}\|  X^J u\|_{L^{q_2}(\Omega)}\geq \cdots\geq C\|u\|_{L^{q_k}(\Omega)}=C\|u\|_{L^q (\Omega)}
\end{aligned}
\]
holds for any $u\in C_{0}^{\infty}(\Omega)$, which implies \eqref{1-8}.

We next prove \eqref{1-9}, in which the case $kp=\tilde{\nu}$ and  $1\leq q<\infty$. If $k=1$, then $p=\tilde{\nu}\geq 2$. For any $q\in[1,+\infty)$, there exists a positive number $r$ such that $1<r<\tilde{\nu}=p$ and $q\leq r^*:=\frac{\tilde{\nu}r}{\tilde{\nu}-r}$. Then,  \eqref{1-8} derives that for any $u\in C_{0}^{\infty}(\Omega)$, we have
\begin{equation}\label{4-14}
\|u\|_{L^q(\Omega)}\leq C\|u\|_{L^{r^*}(\Omega)}\leq C\|Xu\|_{L^r(\Omega)}\leq C\|Xu\|_{L^p(\Omega)}\leq C\sum_{|\alpha|= 1}\|X^{\alpha}u\|_{L^p(\Omega)}.
\end{equation}
Assume that $k\geq 2$ and $kp=\tilde{\nu}\geq 2$. For any $q\in[1,+\infty)$, there exists a positive number $r$ such that $1\leq r<\tilde{\nu}$ and $q\leq r^*:=\frac{\tilde{\nu}r}{\tilde{\nu}-r}$. Note that $\frac{1}{\tilde{\nu}}=\frac{1}{p}-\frac{k-1}{\tilde{\nu}}$ and $p(k-1)<\tilde{\nu}$. It follows from \eqref{1-8} that for any $u\in C_{0}^{\infty}(\Omega)$, we have
\begin{equation}\label{4-15}
\begin{aligned}
\|u\|_{L^{q}(\Omega)}&\leq C\|u\|_{L^{r^{*}}(\Omega)}\leq C\|Xu\|_{L^{r}(\Omega)}\leq C \|Xu\|_{L^{\tilde{\nu}}(\Omega)}\leq C\sum_{|\beta|=k-1}\sum_{j=1}^{m}\|X^{\beta}X_{j}{u}\|_{L^{p}(\Omega)}\\
&\leq C\sum_{|\alpha|=k}\|X^{\alpha}u\|_{L^{p}(\Omega)}.
\end{aligned}
\end{equation}
Consequently, \eqref{1-9} follows from \eqref{4-14} and \eqref{4-15}.
\end{proof}

\subsection{Proof of Theorem \ref{thm2}}
\begin{proof}[Proof of Theorem \ref{thm2}]
For the positive parameters $p,k,\alpha,\beta,s_{1},s_{2}$ satisfying \eqref{conditions},  we have
\begin{equation*}
  \frac{\frac{p}{bs_{2}}}{\frac{p\tilde{\nu}}{\tilde{\nu}-pk}}+\frac{\frac{as_{1}}{bs_{2}}}{s_{1}}
  =\frac{b-a}{bs_{2}}+\frac{a}{bs_{2}}= \frac{1}{s_{2}}~~~~\text{and}~~~~\frac{as_{1}}{bs_{2}}+\frac{p}{bs_{2}}=1.
\end{equation*}
Denoting by $w=\frac{\frac{p\tilde{\nu}}{\tilde{\nu}-pk}}{\frac{p}{bs_{2}}}$ and $r=\frac{s_{1}}{\frac{as_{1}}{bs_{2}}}$, it follows that $\frac{1}{w}+\frac{1}{r}=\frac{1}{s_{2}}$. By \eqref{1-8} and H\"{o}lder's inequality, we have for all $f\in \mathcal{W}_{X,0}^{k,p}(\Omega)\cap L^{s_{1}}(\Omega)$,
\begin{equation*}
  \begin{aligned}
\|f\|_{L^{s_{2}}(\Omega)}^{bs_{2}}&\leq \left(\||f|^{\frac{p}{bs_{2}}}\|_{L^w(\Omega)}
  \||f|^{\frac{as_{1}}{bs_{2}}}\|_{L^r(\Omega)}\right)^{bs_{2}}\\
  &=\left(\|f\|_{L^{\frac{p\tilde{\nu}}{\tilde{\nu}-pk}}(\Omega)}^{\frac{p}{bs_{2}}}
  \|f\|_{L^{s_{1}}(\Omega)}^{\frac{as_{1}}{bs_{2}}} \right)^{bs_{2}}=\|f\|_{L^{\frac{p\tilde{\nu}}{\tilde{\nu}-pk}}(\Omega)}^{p}
  \|f\|_{L^{s_{1}}(\Omega)}^{as_{1}}\\
  &\leq  \left(C\sum_{|J|=k}\|X^{J}f\|_{L^{p}(\Omega)}\right)^{p} \|f\|_{L^{s_{1}}(\Omega)}^{as_{1}}\leq C\sum_{|J|=k}\|X^{J}f\|_{L^{p}(\Omega)}^{p}\|f\|_{L^{s_{1}}(\Omega)}^{as_{1}},
  \end{aligned}
\end{equation*}
where $C>0$ is a positive constant.
\end{proof}

\subsection{Proof of Theorem \ref{thm3}}

\begin{proof}[Proof of Theorem \ref{thm3}]
For any bounded open set $E\subset\subset \Omega$ with $C^{1}$ boundary $\partial E$, we let
\[ r(x):=\inf_{y\in \overline{E}}|x-y|\qquad \forall x\in \Omega, \]
 and set $\delta_0:={\rm dist}(E,\partial \Omega)=\inf_{x\in E,y\in \partial\Omega}|x-y|$. Then, for any $0<\delta<\frac{\delta_{0}}{2}$, we define the function
\[  u_{\delta}(x):=\left(1-\frac{r(x)}{\delta}\right)_{+}\qquad \forall x\in \Omega. \]
 Clearly, $u_{\delta}$ is a Lipschitz function with compact support ${\rm supp}~ u_{\delta}\subset \Omega$, and $u_{\delta}\equiv1$ on $\overline{E}$.  Thus, \cite[Theorem 4.1]{Heinonen2005} indicates that  $u_{\delta}\in W^{1,\infty}(\Omega)\subset \mathcal{W}_{X}^{1,1}(\Omega)$. According to Proposition \ref{prop2-7}, we further have $u_{\delta}\in \mathcal{W}_{X,0}^{1,1}(\Omega)$. Then, for any $0<\lambda<1$,  using \eqref{1-8} we have
\begin{equation}
\begin{aligned}\label{4-16}
|E|^{\frac{\tilde{\nu}-1}{\tilde{\nu}}}&\leq |\{x\in \Omega|u_{\delta}(x)>\lambda\}|^{\frac{\tilde{\nu}-1}{\tilde{\nu}}}\leq \frac{1}{\lambda}\|u_{\delta}\|_{L^{\frac{\tilde{\nu}}{\tilde{\nu}-1}}(\Omega)}\\
&\leq \frac{C}{\lambda}\|Xu_{\delta}\|_{L^{1}(\Omega)}=\frac{C}{\lambda\delta}\int_{\{x\in\Omega|0<r(x)< \delta\}}|Xr|dx.
\end{aligned}
\end{equation}
Letting $\lambda\to 1^{-}$ in \eqref{4-16} and employing the co-area formula, we get
\begin{equation}
\begin{aligned}\label{4-17}
|E|^{\frac{\tilde{\nu}-1}{\tilde{\nu}}}&\leq \frac{C}{\delta}\int_{0}^{\delta}\int_{\{x\in\Omega|r(x)=t\}}|Xr||\nabla r|^{-1}dH_{n-1}dt\\
&=\frac{C}{\delta}\int_{0}^{\delta}\int_{\{x\in\Omega|\tilde{r}(x)=t\}}\left(\sum_{j=1}^{m}\left\langle X_{j}I,\frac{\nabla \tilde{r}}{|\nabla \tilde{r}|}\right\rangle^{2}\right)^{\frac{1}{2}}dH_{n-1}dt,
\end{aligned}
\end{equation}
where $\tilde{r}(x):=\inf_{y\in \partial E}|x-y|$, $\langle \cdot,\cdot\rangle$ denotes the inner product in $\mathbb{R}^n$,
and $H_{n-1}$ is the $(n-1)$-dimensional Hausdorff measure. Letting $\delta\to 0^{+}$ in \eqref{4-17}, we deduce from \cite[Theorem 5.1.3]{Monti2001} and \eqref{def-px} that
\begin{equation}\label{4-18}
  |E|^{\frac{\tilde{\nu}-1}{\tilde{\nu}}}\leq
  C\int_{\partial E}\left(\sum_{j=1}^{m}\left\langle X_{j}I,\eta\right\rangle^{2}\right)^{\frac{1}{2}}dH_{n-1}=C{\rm  Var_{X}}(\chi_{E};\Omega)=CP_{X}(E;\Omega),
\end{equation}
where $\eta=(\eta_{1},\eta_{2},\ldots,\eta_{n})$ denotes the outward unit normal to the $\partial E$.
\end{proof}

\subsection{Proof of Theorem \ref{thm4}}
\begin{proof}[Proof of Theorem \ref{thm4}]
Consider the bilinear form
\begin{equation}\label{4-19}
\mathcal{Q}(u,v):=\int_{\Omega} Xu\cdot Xv dx,~~~~\forall u,v\in \mathcal{H}_{X,0}^1(\Omega).
\end{equation}
 Clearly, $\mathcal{H}_{X,0}^1(\Omega)$ is a dense subspace in $L^2(\Omega)$, and $\mathcal{Q}(\cdot,\cdot)$ is a bilinear, symmetric, non-negative definite, closed functional on $\mathcal{H}_{X,0}^1(\Omega)\times \mathcal{H}_{X,0}^1(\Omega)$. For any $u\in \mathcal{H}_{X,0}^1(\Omega)$, we denote by $\tilde{u}:=\max\{\min\{u,1\},0\}$. It follows from Proposition \ref{prop2-9} that $\tilde{u}\in \mathcal{H}_{X,0}^{1}(\Omega)$ and
  \[ \mathcal{Q}(\tilde{u},\tilde{u})=\int_{\Omega}|X\tilde{u}|^2dx=\int_{\{x\in\Omega|0\leq u(x)\leq 1\}}|Xu|^2dx \leq\int_{\Omega}|Xu|^2dx= \mathcal{Q}(u,u).\]
Thus, $(\mathcal{Q}, \mathcal{H}_{X,0}^1(\Omega))$ is a Dirichlet form. According to  \cite[Theorem 1.3.1, Corollary 1.3.1]{Fukushima1994} and Proposition \ref{prop2-6}, we know that $(\mathcal{Q}, \mathcal{H}_{X,0}^1(\Omega))$ admits a unique generator $\mathcal{L}_{\Omega}$, which is a positive defined self-adjoint operator in $L^2(\Omega)$ with $\text{dom}(\mathcal{L}_{\Omega})\subset \mathcal{H}_{X,0}^{1}(\Omega)$ such that
\begin{equation}\label{4-20}
  \mathcal{Q}(u,v)=(\mathcal{L}_{\Omega}u,v)_{L^2(\Omega)}
\end{equation}
for all $u\in \text{dom}(\mathcal{L}_{\Omega})$ and $v\in \mathcal{H}_{X,0}^{1}(\Omega)$. The operator $\mathcal{L}_{\Omega}$ is the unique self-adjoint extension of $-\triangle_{X}|_{\mathcal{D}(\Omega)}$ with the domain
\[ \begin{aligned}
\text{dom}(\mathcal{L}_{\Omega})&=\{u\in \mathcal{H}_{X,0}^{1}(\Omega)|\exists c\geq 0~\mbox{such that}~|\mathcal{Q}(u,v)|\leq c\|v\|_{L^{2}(\Omega)}~\forall v\in \mathcal{H}_{X,0}^{1}(\Omega)\}\\
&=\{u\in \mathcal{H}_{X,0}^{1}(\Omega)|\triangle_{X} u\in L^2(\Omega)\},
\end{aligned}\]
where $\triangle_{X}=-\sum_{i=1}^{m} X_i^*X_i$ is the H\"{o}rmander operator generated by $X=(X_{1},X_{2},\ldots,X_{m})$.

The spectral theorem (see \cite[Appendix A.5.4]{Grigoryan2009}) allows us to define the corresponding heat semigroup $\{P^{\Omega}_{t}\}_{t\geq 0}$ of $\mathcal{L}_{\Omega}$ such that
\[ P_{t}^{\Omega}:=e^{-t\mathcal{L}_{\Omega}}=\int_{0}^{+\infty}e^{-t\lambda}dE_{\lambda},\]
where $E_{\lambda}$ denotes the spectral resolution of $\mathcal{L}_{\Omega}$  in $L^2(\Omega)$. It follows from \cite[Section 2, p. 509]{Grigoryan2014} that $\{P^{\Omega}_{t}\}_{t\geq 0}$ is a symmetric Markov semigroup. Moreover, by utilizing \cite[Theorem 2.4.2]{Davies1990} and the Sobolev inequality \eqref{1-8} with $p=2$ and $k=1$, we obtain that
\begin{equation}\label{4-21}
\|P^\Omega_{t}f\|_{L^{\infty}(\Omega)}\leq C_{0}t^{-\frac{\tilde{\nu}}{4}}\|f\|_{L^2(\Omega)} \quad\forall t>0,~ f\in L^2(\Omega),
\end{equation}
which yields the ultracontractivity of $\{P^\Omega_t\}_{t\geq 0}$. Here, $C_{0}>0$ is a positive constant. According to \cite[Theorem 2.2.3]{Davies1990} and \eqref{4-21}, we obtain for any $\varepsilon>0$ and $u\in \mathcal{H}_{X,0}^1(\Omega)\cap L^\infty(\Omega)$
\begin{equation}\label{4-22}
\int_{\Omega}|u|^2\ln |u|dx\leq \varepsilon\|Xu\|_{L^2(\Omega)}^2+M(\varepsilon)\|u\|_{L^2(\Omega)}^2+\|u\|_{L^2(\Omega)}^2\ln\|u\|_{L^2(\Omega)},
\end{equation}
where $M(\varepsilon)=\ln C_0-\frac{\tilde{\nu}}{4}\ln \varepsilon$ with $C_0$ in \eqref{4-21}.

Assuming that $p\geq1$,  $k\in \mathbb{N}^{+}$ with $kp\leq \tilde{\nu}$,  we set $$
q=\left\{
            \begin{array}{ll}
             \frac{p\tilde{\nu}}{\tilde{\nu}-pk} , & \hbox{if $kp<\tilde{\nu}$;} \\[3mm]
              \gamma p,& \hbox{if $kp=\tilde{\nu}$;}
            \end{array}
          \right.
$$ for some $\gamma>1$. For any $u\in \mathcal{W}^{k,p}_{X,0}(\Omega)$ with $\|u\|_{L^p(\Omega)}=1$, using Theorem \ref{thm1} and Jensen's inequality, we obtain
\begin{equation}
\begin{aligned}\label{4-23}
   \ln\left( C\sum_{|J|=k}\|X^J u\|_{L^p(\Omega)}\right) &
     \geq\ln\|u\|_{L^q(\Omega)}=\frac{1}{q}\ln\left(\int_{\Omega}|u|^{q-p}|u|^pdx\right)\geq  \left(\frac{1}{p}-\frac{1}{q}\right)\int_\Omega|u|^p\ln|u|^p dx\\
&=\left\{
    \begin{array}{ll}
      \frac{k}{\tilde{\nu}}\int_\Omega|u|^p\ln|u|^p dx, & \hbox{if $kp<\tilde{\nu}$;} \\[3mm]
    \frac{k}{\tilde{\nu}}\left(1-\frac{1}{\gamma} \right)\int_\Omega|u|^p\ln|u|^p dx, & \hbox{if $kp=\tilde{\nu}$.}
    \end{array}
  \right.
  \end{aligned}
\end{equation}
Taking $\gamma\to \infty$ in \eqref{4-23}, we derive the second inequality in \eqref{1-16}, while the first inequality in \eqref{1-16} is due to
 $\ln Cx\leq \frac{C}{e}x$ for $C>0$ and $x>0$.
\end{proof}

\subsection{Proof of Theorem \ref{thm5}}

Next, we devote to Theorem \ref{thm5}. The proof of Theorem \ref{thm5} is based on the following lemma.

\begin{lemma}\label{lemma4-3}
Let $X=(X_1,X_{2},\ldots,X_m)$ satisfy condition (H). Suppose that $\Omega\subset\subset U$ is a bounded open subset of $U$, and $p>\tilde{\nu}$ is a positive real number. Then, there exists a positive constant $C>0$ such that for any $u\in C_0^\infty(\Omega)$,
\begin{equation}\label{4-24}
|u(x)|\leq C\|Xu\|_{L^p(\Omega)}\qquad \forall x\in \overline{\Omega},
\end{equation}
and
\begin{equation}\label{4-25}
|u(x)-u(y)|\leq Cd^{1-\frac{\tilde{\nu}}{p}}(x,y)\|Xu\|_{L^p(\Omega)} \qquad \forall x,y\in \overline{\Omega},
\end{equation}
where $C>0$ is a positive constant.
\end{lemma}
\begin{proof}
By Proposition \ref{prop2-5}, we deduce that for any $x\in\overline{\Omega}$ and $0<\varepsilon\leq \rho_{\overline{\Omega}}$, we have
\begin{equation}\label{4-26}
  \int_{B(x,\varepsilon)}\frac{d(x,y)^{\frac{p}{p-1}}}{|B(x,d(x,y))|^{\frac{p}{p-1}}}dy\leq C\varepsilon^{\frac{p-\tilde{\nu}}{p-1}},
\end{equation}
where $C>0$ is a positive constant independent of $\varepsilon$. Taking $\varepsilon=\frac{\rho_{\overline{\Omega}}}{2}$ in \eqref{4-26} and using Proposition \ref{prop2-2} and Proposition \ref{prop2-3}, we obtain that for any $x\in \overline{\Omega}$,
\begin{equation}
\begin{aligned}\label{4-27}
\int_{\Omega}\frac{d(x,y)^{\frac{p}{p-1}}}{|B(x,d(x,y))|^{\frac{p}{p-1}}}dy&\leq\int_{\Omega\setminus B(x,\frac{\rho_{\overline{\Omega}}}{2})}\frac{d(x,y)^{\frac{p}{p-1}}}{|B(x,d(x,y))|^{\frac{p}{p-1}}}dy+\int_{ B(x,\frac{\rho_{\overline{\Omega}}}{2})}\frac{d(x,y)^{\frac{p}{p-1}}}{|B(x,d(x,y))|^{\frac{p}{p-1}}}dy\\
&\leq  C\left(\left( \frac{\rho_{\overline{\Omega}}}{2}\right)^{\frac{p-\tilde{\nu}}{p-1}}+\frac{|\Omega|\delta_{\overline{\Omega}}^{\frac{p}{p-1}}}{\left( \frac{\rho_{\overline{\Omega}}}{2}\right)^{\frac{p\tilde{\nu}}{p-1}}}\right),
\end{aligned}
\end{equation}
where $\delta_{\overline{\Omega}}=\sup_{x,y\in\overline{\Omega}}d(x,y)$. Then, by Proposition \ref{prop2-6}, Proposition \ref{prop3-2} and \eqref{4-27}, we get
  \begin{equation*}
    \begin{aligned}
    |u(x)|&\leq C\int_{\Omega}\frac{d(x,y)}{|B(x,d(x,y))|}(|Xu(y)|+|u(y)|)dy\\
    &\leq C\left(\|Xu\|_{L^p(\Omega)}+\|u\|_{L^p(\Omega)}\right)\left(\int_{\Omega}\frac{d(x,y)^{\frac{p}{p-1}}}{|B(x,d(x,y))|^{\frac{p}{p-1}}}dy\right)^{\frac{p-1}{p}}\\
    &\leq C\|Xu\|_{L^p(\Omega)}\qquad \forall x\in \overline{\Omega},
    \end{aligned}
  \end{equation*}
which yields \eqref{4-24}.

We next prove \eqref{4-25}. By Proposition \ref{prop3-1}, there exist positive constants $C>0$, $\alpha_{1}\geq 1$ and $r_0>0$ such that for any $x_0\in \overline{\Omega}$ and $0<r<r_0$, we have
\begin{equation}\label{4-28}
  |f(x)-f_{B}|\leq C\int_{B(x_{0},\alpha_{1}r)}\frac{d(x,y)}{|B(x,d(x,y))|}|Xf(y)|dy\qquad \forall x\in B(x_0,r),
\end{equation}
and any $f\in C^{\infty}(B(x_0,\alpha_{1}r))$, where $\alpha_{1}\geq 1$ is a positive constant independent of $f$ and $B(x_0,r)$, and $f_{B}=|B(x_0,r)|^{-1}\int_{B(x_0,r)}f(y)dy$. Then, for any $u\in C_{0}^{\infty}(\Omega)$ and any $x,y\in \overline{\Omega}$, we denote by $\delta:=d(x,y)$. The proof of \eqref{4-25} will be divided into the following two cases:\\
\textbf{\emph{Case 1:}} $0<\delta<\min\left\{\frac{\rho_{\overline{\Omega}}}{\alpha_{1}+1},r_{0}\right\}$ with $\rho_{\overline{\Omega}}$ being the positive constant determined by Proposition \ref{prop2-3}. Let $x_{0}\in B(x,\delta)\cap B(y,\delta)\cap \Omega$ such that $x,y\in B(x_{0},\delta)$. It follows that
\[B(x_{0},\alpha_{1}\delta)\subset B(x,(\alpha_{1}+1)\delta)\cap B(y,(\alpha_{1}+1)\delta).\] For simplification, we set $B_{x_{0}}:=B(x_{0},\delta)$ and $T(y,z):=\frac{d(y,z)}{|B(y,d(y,z))|}$. Therefore, \eqref{2-16} and \eqref{4-28} give that
\begin{equation}\label{4-29}
\begin{aligned}
&|u(x)-u(y)|\leq |u(x)-u_{B_{x_{0}}}|+|u(y)-u_{B_{x_{0}}}|\\
    &\leq C\int_{B(x_{0},\alpha_{1}\delta)}T(x,z)|Xu(z)|dz
    +C\int_{B(x_{0},\alpha_{1}\delta)}T(y,z)|Xu(z)|dz\\
    &\leq C\int_{B(x,(\alpha_{1}+1)\delta)}T(x,z)|Xu(z)|dz
    +C\int_{B(y,(\alpha_{1}+1)\delta)}T(y,z)|Xu(z)|dz\\
    &\leq C\left[\left(\int_{B(x,(\alpha_{1}+1)\delta)}T(x,z)^{\frac{p}{p-1}} dz\right)^{\frac{p-1}{p}}
    +\left(\int_{B(y,(\alpha_{1}+1)\delta)}T(y,z)^{\frac{p}{p-1}}  dz\right)^{\frac{p-1}{p}}\right]
    \|Xu\|_{L^p(\Omega)}\\
    &\leq C(\alpha_{1}+1)^{1-\frac{\tilde{\nu}}{p}}\delta^{1-\frac{\tilde{\nu}}{p}}  \|Xu\|_{L^p(\Omega)},
\end{aligned}
\end{equation}
where  $C>0$ is a positive constant independent of $x,y$ and $x_{0}$.\\

\noindent
\textbf{\emph{Case 2:}} $\delta\geq\min\left\{\frac{\rho_{\overline{\Omega}}}{\alpha_{1}+1},r_{0}\right\}$. Employing \eqref{4-24}, we obtain
\begin{equation}\label{4-30}
|u(x)-u(y)|\leq 2C\|Xu\|_{L^p(\Omega)}\leq \frac{2C}{\left(\min\left\{\frac{\rho_{\overline{\Omega}}}{\alpha_{1}+1},r_{0}\right\} \right)^{1-\frac{\tilde{\nu}}{p}}}\delta^{1-\frac{\tilde{\nu}}{p}}\|Xu\|_{L^p(\Omega)},
\end{equation}
where $C>0$ is a positive constant independent of $x$ and $y$.

As a result of \eqref{4-29} and \eqref{4-30}, we have for any $u\in C_{0}^{\infty}(\Omega)$,
\begin{equation*}
|u(x)-u(y)|\leq Cd(x,y)^{1-\frac{\tilde{\nu}}{p}}\|Xu\|_{L^p(\Omega)}\qquad \forall x,y\in \overline{\Omega},
\end{equation*}
which gives \eqref{4-25}.
\end{proof}

Now, we prove Theorem \ref{thm5}.
\begin{proof}[Proof of Theorem \ref{thm5}]
Obviously, Lemma \ref{lemma4-3} implies that for $p>\tilde{\nu}$,
\begin{equation}\label{4-31}
\|u\|_{\mathcal{S}^{0,1-\frac{\tilde{\nu}}{p}}(\overline{\Omega})}\leq C \|Xu\|_{L^p(\Omega)}
\end{equation}
holds for any $u\in \mathcal{W}^{1,p}_{X,0}(\Omega)$, which yields  \eqref{holder-embedding} for $k=1$ since $\frac{\tilde{\nu}}{p}\notin \mathbb{N}^{+}$.

Let us consider the case $k\geq 2$. If $\frac{\tilde{\nu}}{p}\notin \mathbb{N}^{+}$, we set $l=\left[\frac{\tilde{\nu}}{p}\right]$. It follows that $lp<\tilde{\nu}<kp$ and $k\geq l+1$. Denote by $r=\frac{p\tilde{\nu}}{\tilde{\nu}-lp}>\tilde{\nu} $. According to Theorem \ref{thm1} and \eqref{4-31} we have for any $u\in C_{0}^{\infty}(\Omega)$,
  \begin{equation}\label{4-32}
    \begin{aligned}
    \sum_{l+1\leq|J|\leq k}\|X^{J}u\|_{L^p(\Omega)}&\geq C\sum_{1\leq|J|\leq k-l}\|X^{J}u\|_{L^r(\Omega)}\geq C\sum_{0\leq|J|\leq k-l-1}\|X^{J}u\|_{\mathcal{S}^{0,1-\frac{\tilde{\nu}}{r}}(\overline{\Omega})}.
    \end{aligned}
  \end{equation}
Combining \eqref{4-32} and Proposition \ref{prop2-6}, we get
  \[\|u\|_{\mathcal{S}^{k-\left[\frac{\tilde{\nu}}{p}\right]-1,\left[\frac{\tilde{\nu}}{p}\right]+1-\frac{\tilde{\nu}}{p}}(\overline{\Omega})}
  \leq C\sum_{|J|= k}\|X^{J}u\|_{L^p(\Omega)}\qquad \forall u\in C_{0}^{\infty}(\Omega).\]

When $\frac{\tilde{\nu}}{p}\in \mathbb{N}^{+}$, we let $l=\frac{\tilde{\nu}}{p}-1\geq 0$. As a result, $k-l-1=k-\frac{\tilde{\nu}}{p}\geq 1$. Denote by $r:=\frac{p\tilde{\nu}}{\tilde{\nu}-lp}=\tilde{\nu}$ and choose a positive number $q>\tilde{\nu}$. By Theorem \ref{thm1} and \eqref{4-31} we have for any $u\in C_{0}^{\infty}(\Omega)$,
  \begin{equation*}
    \begin{aligned}
    \sum_{l+2\leq|J|\leq k}\|X^{J}u\|_{L^p(\Omega)}&\geq C\sum_{2\leq|J|\leq k-l}\|X^{J}u\|_{L^r(\Omega)}\geq C\sum_{1\leq|J|\leq k-l-1}\|X^{J}u\|_{L^q(\Omega)}\\
    &\geq C\sum_{0\leq|J|\leq k-l-2}\|X^{J}u\|_{\mathcal{S}^{0,1-\frac{\tilde{\nu}}{q}}(\overline{\Omega})},
    \end{aligned}
  \end{equation*}
  which, combining with the Proposition \ref{prop2-6}, derives that
  \[\|u\|_{\mathcal{S}^{k-\frac{\tilde{\nu}}{p}-1,\epsilon}(\overline{\Omega})}
  \leq C \sum_{|\alpha|= k}\|X^{\alpha}u\|_{L^p(\Omega)}\qquad \forall u\in C_{0}^{\infty}(\Omega),\]
  where $\epsilon=1-\frac{\tilde{\nu}}{q}\in(0,1)$.
\end{proof}

\subsection{Proof of Theorem \ref{thm6}}

We next introduce the following abstract version of the Rellich-Kondrachov compactness theorem, which will be used in constructing the compact embedding result.
\begin{proposition}[{\cite[Theorem 4]{Hajlasz1998}}]
\label{prop4-1}
Let $Y$ be a set equipped with a finite measure $\mu$. Assume that a linear normed space $G$ of measurable functions on $Y$ has the following two properties:
\begin{enumerate}
  \item [(1)] There exists a constant $q>1$ such that the embedding $G\hookrightarrow L^q(Y,\mu)$ is bounded;
  \item [(2)] Every bounded sequence in $G$ contains a subsequence that converges almost everywhere.
\end{enumerate}
 Then the embedding $G\hookrightarrow L^s(Y,\mu)$ is compact for every $1\leq s<q$.
\end{proposition}

\begin{proof}[Proof of Theorem \ref{thm6}]
By Theorem \ref{thm1},  for $p\geq 1$, $k\in \mathbb{N}^{+}$ with $kp<\tilde{\nu}$, the embedding
$  \mathcal{W}_{X,0}^{k,p}(\Omega)\hookrightarrow L^{q}(\Omega)$ is continuous for $q=\frac{\tilde{\nu}p}{\tilde{\nu}-kp}>1$. According to \cite[Corollary 3.3]{Danielli1991}, $\mathcal{W}_{X,0}^{1,p}(\Omega)$ is compactly embedded into $L^p(\Omega)$ for $p\geq 1$. Thus, for any bounded sequence $\{u_k\}_{k=1}^{\infty}$ in $\mathcal{W}_{X,0}^{k,p}(\Omega)\subset \mathcal{W}_{X,0}^{1,p}(\Omega)$, there exists a subsequence $\{u_{k_{j}}\}_{j=1}^{\infty}\subset\{u_k\}_{k=1}^{\infty}$ such that $u_{k_{j}}\to u$ in $L^{p}(\Omega)$. The Riesz theorem allows us to find a subsequence $\{v_{j}\}_{j=1}^{\infty}\subset  \{u_{k_{j}}\}_{j=1}^{\infty}$  such that $v_{j}\to u$ almost everywhere on $\Omega$ as $j\to +\infty$. Hence, we conclude from Proposition \ref{prop4-1} that the embedding $\mathcal{W}_{X,0}^{k,p}(\Omega)\hookrightarrow L^s(\Omega)$ is compact for $s\in [1,\frac{\tilde{\nu}p}{\tilde{\nu}-kp})$.
\end{proof}

\subsection{Proof of Theorem \ref{thm7}}

\begin{proof}[Proof of Theorem \ref{thm7}]

Assuming that $u\in \mathcal{W}^{1,\tilde{\nu}}_{X,0}(\Omega)$ with $\|Xu\|_{L^{\tilde{\nu}}(\Omega)}\leq 1$, we choose an approximating sequence $\{u_{k}\}_{k=1}^{\infty}\subset C_{0}^{\infty}(\Omega)$ such that $u_{k}\to u$ in $\mathcal{W}^{1,\tilde{\nu}}_{X,0}(\Omega)$ as $k\to +\infty$. For $p>1$, $q=\frac{p}{p-1}$, and $v\in L^{q}(\Omega)$ with $\|v\|_{L^q(\Omega)}\leq 1$,  by  Proposition \ref{prop3-2} and H\"{o}lder's inequality we have
\begin{equation}\label{4-33}
\begin{aligned}
&\int_{\Omega}|u_{k}(x)v(x)|dx\leq C_{0}\int_{\Omega}\int_{\Omega}T(x,y)|v(x)|(|Xu_{k}(y)|+|u_{k}(y)|)dydx\\
&= C_{0}\int_{\Omega}\int_{\Omega}(T(x,y)^{\frac{1}{\tilde{\nu}p}}(|Xu_{k}(y)|+|u_{k}(y)|)|v(x)|^{\frac{1}{\tilde{\nu}}})
(T(x,y)^{1-\frac{1}{\tilde{\nu}p}}|v(x)|^{\frac{\tilde{\nu}-1}{\tilde{\nu}}})dydx\\
&\leq  C_{0}\left(\int_{\Omega}\int_{\Omega}T(x,y)^{\frac{1}{p}}(|Xu_{k}(y)|+|u_{k}(y)|)^{\tilde{\nu}}|v(x)|dydx\right)^{\frac{1}{\tilde{\nu}}}
\left(\int_{\Omega}\int_{\Omega}T(x,y)^{\frac{\tilde{\nu}p-1}{\tilde{\nu}p-p}}|v(x)|dydx\right)^{\frac{\tilde{\nu}-1}{\tilde{\nu}}},
\end{aligned}
\end{equation}
where $T(x,y)=\frac{d(x,y)}{|B(x,d(x,y))|}$ and $ C_{0}>0$ is a positive constant independent of $u_{k}$. From Proposition \ref{prop2-4} and Proposition \ref{prop2-5}, we have
\begin{equation}\label{4-34}
\begin{aligned}
&\int_{\Omega}\int_{\Omega}T(x,y)^{\frac{1}{p}}(|Xu_{k}(y)|+|u_{k}(y)|)^{\tilde{\nu}}|v(x)|dydx\\
&\leq  \int_{\Omega}(|Xu_{k}(y)|+|u_{k}(y)|)^{\tilde{\nu}}\left( \int_{\Omega}\frac{d(x,y)}{|B(x,d(x,y))|}dx\right)^{\frac{1}{p}}\left(\int_{\Omega}|v|^{q}dx\right)^{\frac{1}{q}}dy\\
&\leq 2^{\tilde{\nu}-1}\|u_{k}\|_{\mathcal{W}_{X,0}^{1,\tilde{\nu}}(\Omega)}^{\tilde{\nu}}C_{3}^{\frac{1}{p}} \left( \int_{\Omega}\frac{d(y,x)}{|B(y,d(y,x))|}dx\right)^{\frac{1}{p}}\\
&\leq \left(\frac{C_{1}+C_{2}}{C_{1}}\right)^{\frac{1}{p}}\cdot \frac{2^{\tilde{\nu}-1+\frac{1}{p}}C_{3}^{\frac{2}{p}}}{(CC_{1})^{\frac{1}{\tilde{\nu}p}}}|\Omega|^{\frac{1}{p\tilde{\nu}}}\|u_{k}\|_{\mathcal{W}_{X,0}^{1,\tilde{\nu}}(\Omega)}^{\tilde{\nu}} ,
\end{aligned}
\end{equation}
where $C,C_{1},C_{2},C_{3}$ are the positive constants appeared in Proposition \ref{prop2-2}-Proposition \ref{prop2-4}. Moreover, applying Proposition \ref{prop2-5} for $\mu=\eta=\frac{\tilde{\nu}p-1}{\tilde{\nu}p-p}$ and $\xi=\frac{1}{p}\in(0,1)$, we have
\begin{equation}\label{4-35}
\begin{aligned}
&\int_{\Omega}\int_{\Omega}T(x,y)^{\frac{\tilde{\nu}p-1}{\tilde{\nu}p-p}}|v(x)|dydx=\int_{\Omega}|v(x)|\left(\int_{\Omega}\frac{d(x,y)^{\frac{\tilde{\nu}p-1}{\tilde{\nu}p-p}}}{|B(x,d(x,y))|^{\frac{\tilde{\nu}p-1}{\tilde{\nu}p-p}}}dy\right)dx\\
&\leq \left[\left(\frac{C_{2}}{C_{1}}\right)^{\frac{\tilde{\nu}p-1}{p(\tilde{\nu}-1)}}+1\right]\frac{C_{3}}{(CC_{1})^{\frac{\tilde{\nu}p-1}{p\tilde{\nu}(\tilde{\nu}-1)}}}\frac{2^{\frac{1}{p}}}{2^{\frac{1}{p}}-1}|\Omega|^\frac{1}{\tilde{\nu}p}\int_{\Omega}|v(x)|dx\\
&\leq \left[\left(\frac{C_{2}}{C_{1}}\right)^{\frac{\tilde{\nu}p-1}{p(\tilde{\nu}-1)}}+1\right]\frac{C_{3}}{(CC_{1})^{\frac{\tilde{\nu}p-1}{p\tilde{\nu}(\tilde{\nu}-1)}}}\frac{2^{\frac{1}{p}}}{2^{\frac{1}{p}}-1}|\Omega|^\frac{1}{\tilde{\nu}p}\left(\int_{\Omega}|v(x)|^{q}dx\right)^{\frac{1}{q}}|\Omega|^{1-\frac{1}{q}}\\
&\leq\left[\left(\frac{C_{2}}{C_{1}}\right)^{\frac{\tilde{\nu}p-1}{p(\tilde{\nu}-1)}}+1\right]\frac{C_{3}}{(CC_{1})^{\frac{\tilde{\nu}p-1}{p\tilde{\nu}(\tilde{\nu}-1)}}}\frac{2^{\frac{1}{p}}}{2^{\frac{1}{p}}-1}|\Omega|^{\frac{1}{p\tilde{\nu}}+\frac{1}{p}}.
\end{aligned}
\end{equation}
Hence, \eqref{4-33}-\eqref{4-35} derive that for any $p>1$ and $k\geq 1$,
\begin{equation}\label{4-36}
\begin{aligned}
\|u_{k}\|_{L^p(\Omega)}&=\sup_{v\in L^q(\Omega),~\|v\|_{L^q(\Omega)}\leq 1}\int_{\Omega}|u_{k}(x)v(x)|dx\\
&\leq C_{0}\left(1+\frac{C_{2}}{C_{1}}\right)^{\frac{1}{\tilde{\nu}p}}\left(1+\left(\frac{C_{2}}{C_{1}}\right)^{\frac{\tilde{\nu}p-1}{p(\tilde{\nu}-1)}}\right)^{\frac{\tilde{\nu}-1}{\tilde{\nu}}}
\frac{2^{1+\frac{1}{p}-\frac{1}{\tilde{\nu}}}}{(2^{\frac{1}{p}}-1)^{1-\frac{1}{\tilde{\nu}}}}\frac{C_{3}^{1-\frac{1}{\tilde{\nu}}+\frac{2}{\tilde{\nu}p}}}{(CC_{1})^{\frac{1}{\tilde{\nu}}}}
\|u_{k}\|_{\mathcal{W}_{X,0}^{1,\tilde{\nu}}(\Omega)}|\Omega|^{\frac{1}{p}}.
\end{aligned}
\end{equation}
Observing that $u_{k}\to u$ in $\mathcal{W}^{1,\tilde{\nu}}_{X,0}(\Omega)$, by \eqref{1-9} we have
\begin{equation*}
 |\|u_{k}\|_{L^p(\Omega)}-\|u\|_{L^p(\Omega)}|\leq \|u-u_{k}\|_{L^p(\Omega)}\leq C\|Xu-Xu_{k}\|_{\mathcal{W}^{1,\tilde{\nu}}_{X,0}(\Omega)}\to 0\quad\mbox{as}\quad k\to\infty.
\end{equation*}
Thus, taking $k\to \infty$ in \eqref{4-36}, we obtain from inequality $2^{\frac{1}{p}}-1\geq \frac{1}{p}2^{\frac{1}{p}-1}$ that
\begin{equation}\label{4-37}
\begin{aligned}
  \|u\|_{L^p(\Omega)}&\leq C_{0}\left(1+\frac{C_{2}}{C_{1}}\right)^{\frac{1}{\tilde{\nu}p}}\left(1+\left(\frac{C_{2}}{C_{1}}\right)^{\frac{\tilde{\nu}p-1}{p(\tilde{\nu}-1)}}\right)^{\frac{\tilde{\nu}-1}{\tilde{\nu}}}
\frac{2^{1+\frac{1}{p}-\frac{1}{\tilde{\nu}}}}{(2^{\frac{1}{p}}-1)^{1-\frac{1}{\tilde{\nu}}}}\frac{C_{3}^{1-\frac{1}{\tilde{\nu}}+\frac{2}{\tilde{\nu}p}}}{(CC_{1})^{\frac{1}{\tilde{\nu}}}}\|u\|_{\mathcal{W}_{X,0}^{1,\tilde{\nu}}(\Omega)} |\Omega|^{\frac{1}{p}}\\
&\leq  C_{0}\left(1+\frac{C_{2}}{C_{1}}\right)\frac{2^{1+\frac{1}{p}-\frac{1}{\tilde{\nu}}}}{(2^{\frac{1}{p}}-1)^{1-\frac{1}{\tilde{\nu}}}}\frac{C_{3}^{1-\frac{1}{\tilde{\nu}}+\frac{2}{\tilde{\nu}p}}}{(CC_{1})^{\frac{1}{\tilde{\nu}}}}\left( 1+\lambda_{1}(\tilde{\nu})^{-1}\right)^{\frac{1}{\tilde{\nu}}} |\Omega|^{\frac{1}{p}}\\
&\leq C_{0}\left(1+\frac{C_{2}}{C_{1}}\right)2^{2+\frac{1}{\tilde{\nu}}(\frac{1}{p}-2)}\frac{C_{3}^{1-\frac{1}{\tilde{\nu}}+\frac{2}{\tilde{\nu}p}}}{(CC_{1})^{\frac{1}{\tilde{\nu}}}}\left( 1+\lambda_{1}(\tilde{\nu})^{-1}\right)^{\frac{1}{\tilde{\nu}}} p^{1-\frac{1}{\tilde{\nu}}} |\Omega|^{\frac{1}{p}}\\
&\leq 4C_{0}\left(1+\frac{C_{2}}{C_{1}}\right)\frac{C_{3}^{1+\frac{1}{\tilde{\nu}}}}{(CC_{1})^{\frac{1}{\tilde{\nu}}}}\left( 1+\lambda_{1}(\tilde{\nu})^{-1}\right)^{\frac{1}{\tilde{\nu}}}p^{1-\frac{1}{\tilde{\nu}}} |\Omega|^{\frac{1}{p}},
\end{aligned}
\end{equation}
where
\begin{equation}\label{4-38}
  \lambda_{1}(\tilde{\nu}):=\inf_{u\in \mathcal{W}_{X,0}^{1,\tilde{\nu}}(\Omega),~u\neq 0}\frac{\int_{\Omega}|Xu|^{\tilde{\nu}}dx}{\int_{\Omega}|u|^{\tilde{\nu}}dx}>0.
\end{equation}

Denote by
\[ C_{4}:=4C_{0}\left(1+\frac{C_{2}}{C_{1}}\right)\frac{C_{3}^{1+\frac{1}{\tilde{\nu}}}}{(CC_{1})^{\frac{1}{\tilde{\nu}}}}\left( 1+\lambda_{1}(\tilde{\nu})^{-1}\right)^{\frac{1}{\tilde{\nu}}}.\]
It follows that
\[ \|u\|_{L^p(\Omega)}^{p}\leq C_{4}^{p}p^{p-\frac{p}{\tilde{\nu}}}|\Omega|\qquad \forall p>1,\]
Consequently, we get
\begin{equation*}
\begin{aligned}
&\int_\Omega\left(e^{\sigma|u|^{\frac{\tilde{\nu}}{\tilde{\nu}-1}}}-1\right)dx
=\sum_{j=1}^\infty \frac{\sigma^j}{j!}\|u\|_{L^{\frac{j\tilde{\nu}}{\tilde{\nu}-1}}(\Omega)}^{\frac{j\tilde{\nu}}{\tilde{\nu}-1}}\leq \sum_{j=1}^\infty \frac{\sigma^j}{j!}C_{4}^{\frac{j\tilde{\nu}}{\tilde{\nu}-1}}\left( \frac{j\tilde{\nu}}{\tilde{\nu}-1}\right)^{j}|\Omega|<+\infty,
\end{aligned}
\end{equation*}
provided
\[ 0<\sigma<\frac{\tilde{\nu}-1}{e\tilde{\nu}C_{4}^{\frac{\tilde{\nu}}{\tilde{\nu}-1}}}=\frac{\tilde{\nu}-1}{e\tilde{\nu}}C_{3}^{-\frac{\tilde{\nu}+1}{\tilde{\nu}-1}}\left( \frac{C_{1}}{4C_{0}(C_{1}+C_{2})}\right)^{\frac{\tilde{\nu}}{\tilde{\nu}-1}}\cdot\left(\frac{\lambda_{1}(\tilde{\nu})CC_{1}}{\lambda_{1}(\tilde{\nu})+1} \right)^{\frac{1}{\tilde{\nu}-1}}. \]

\end{proof}

\section*{Acknowledgements}
Hua Chen is supported by National Natural Science Foundation of China (Grant No. 12131017) and National Key R\&D Program of China (no. 2022YFA1005602). Hong-Ge Chen is supported by National Natural Science Foundation of China (Grant No. 12201607) and Knowledge Innovation Program of Wuhan-Shuguang Project (Grant No. 2023010201020286). Jin-Ning Li is supported by China National Postdoctoral Program for Innovative Talents (Grant No.  BX20230270).

\end{document}